\documentclass[twoside,english,american,DIV12]{scrartcl}
\usepackage[T1]{fontenc}
\usepackage[latin1]{inputenc}
\usepackage{babel}
\usepackage{verbatim}
\usepackage{prettyref}
\usepackage{mathtools}
\usepackage{amsthm}
\usepackage{amsmath}
\usepackage{amssymb}
\usepackage[unicode=true,pdfusetitle,
 bookmarks=true,bookmarksnumbered=false,bookmarksopen=false,
 breaklinks=false,pdfborder={0 0 1},backref=false,colorlinks=false]
 {hyperref}

\makeatletter
 \theoremstyle{definition}
 \newtheorem*{defn*}{\protect\definitionname}
\theoremstyle{plain}
\newtheorem{thm}{\protect\theoremname}[section]
  \theoremstyle{definition}
  \newtheorem{example}[thm]{\protect\examplename}
  \theoremstyle{plain}
  \newtheorem{lem}[thm]{\protect\lemmaname}
  \theoremstyle{remark}
  \newtheorem{rem}[thm]{\protect\remarkname}
  \theoremstyle{plain}
  \newtheorem{prop}[thm]{\protect\propositionname}
  \theoremstyle{plain}
  \newtheorem{cor}[thm]{\protect\corollaryname}

\usepackage[T1]{fontenc}    % Silbentrennung auch nach Umlauten mglich
\usepackage{enumerate}
\usepackage{a4wide}        % A4-Format
\usepackage{amsmath}
\usepackage{euscript}
\usepackage{url}
\usepackage{scrpage2}
\usepackage{tu-preprint}
\usepackage{scrpage2}
\allowdisplaybreaks[1] % allow page breaks in some displayed equations
\usepackage{amsfonts}
\usepackage{amsfonts}
\newenvironment{keywords}{ \noindent\footnotesize\textbf{Keywords and phrases:}}{}

\newenvironment{class}{\noindent\footnotesize\textbf{Mathematics subject classification 2010:}}{}

\usepackage{color}

\newcommand*{\dive}{\operatorname{div}}

\newcommand*{\curl}{\operatorname{curl}}

\newcommand*{\grad}{\operatorname{grad}}

\renewcommand*{\i}{\mathrm{i}}

\newcommand{\loc}{\mathrm{loc}}
\DeclareMathAccent{\Circ}{\mathalpha}{operators}{"17}
\newcommand{\interior}[1]{\Circ{#1}}
\renewcommand{\Im}{\operatorname{\mathfrak{Im}}}
\renewcommand{\Re}{\operatorname{\mathfrak{Re}}}

\newcommand{\m}{\mathrm{m}}
\renewcommand{\hat}{\widehat}
\renewcommand{\tilde}{\widetilde}
\renewcommand*{\epsilon}{\varepsilon}
\renewcommand*{\rho}{\varrho}

\newrefformat{prop}{Proposition \ref{#1}}
\newrefformat{lem}{Lemma \ref{#1}}
\newrefformat{thm}{Theorem \ref{#1}}
\newrefformat{cor}{Corollary \ref{#1}}
\newrefformat{rem}{Remark \ref{#1}}
\newrefformat{ex}{Example \ref{#1}}
\newrefformat{sub}{Subsection \ref{#1}}
\newrefformat{eq}{(\ref{#1})}

\newrefformat{item}{(\ref{#1})}

\theoremstyle{definition}
\newtheorem*{hyp}{Hypotheses}

\makeatother

\institut{Institut f\"ur Analysis}

\preprintnumber{MATH-AN-03-2014}

\preprinttitle{Exponential stability for second order evolutionary problems.}

\author{Sascha Trostorff}

\AtBeginDocument{
  
}

\makeatother

  \addto\captionsamerican{\renewcommand{\corollaryname}{Corollary}}
  \addto\captionsamerican{\renewcommand{\definitionname}{Definition}}
  \addto\captionsamerican{\renewcommand{\examplename}{Example}}
  \addto\captionsamerican{\renewcommand{\lemmaname}{Lemma}}
  \addto\captionsamerican{\renewcommand{\propositionname}{Proposition}}
  \addto\captionsamerican{\renewcommand{\remarkname}{Remark}}
  \addto\captionsamerican{\renewcommand{\theoremname}{Theorem}}
  \addto\captionsenglish{\renewcommand{\corollaryname}{Corollary}}
  \addto\captionsenglish{\renewcommand{\definitionname}{Definition}}
  \addto\captionsenglish{\renewcommand{\examplename}{Example}}
  \addto\captionsenglish{\renewcommand{\lemmaname}{Lemma}}
  \addto\captionsenglish{\renewcommand{\propositionname}{Proposition}}
  \addto\captionsenglish{\renewcommand{\remarkname}{Remark}}
  \addto\captionsenglish{\renewcommand{\theoremname}{Theorem}}
  \providecommand{\corollaryname}{Corollary}
  \providecommand{\definitionname}{Definition}
  \providecommand{\examplename}{Example}
  \providecommand{\lemmaname}{Lemma}
  \providecommand{\propositionname}{Proposition}
  \providecommand{\remarkname}{Remark}
\providecommand{\theoremname}{Theorem}

\begin{document}
\makepreprinttitlepage

\author{ Sascha Trostorff \\ Institut f\"ur Analysis, Fachrichtung Mathematik\\ Technische Universit\"at Dresden\\ Germany\\ sascha.trostorff@tu-dresden.de}

\title{Exponential stability for second order evolutionary problems.}

\maketitle
\begin{abstract} \textbf{Abstract.} We study the exponential stability
of evolutionary equations. The focus is laid on second order problems
and we provide a way to rewrite them as a suitable first order evolutionary
equation, for which the stability can be proved by using frequency
domain methods. The problem class under consideration is broad enough
to cover integro-differential equations, delay-equations and classical
evolution equations within a unified framework. \end{abstract}

\begin{keywords}Exponential stability, hyperbolic problems, evolutionary
equations, integro-differential equations, frequency domain methods,
delay.\end{keywords}

\begin{class} 35B30; 35B40; 35G15; 47N20  \end{class}

\newpage

\tableofcontents{} 

\newpage

\section{Introduction}

Evolutionary equations, as they were first introduced by Picard (\cite{Picard,Picard_McGhee,Picard2014_survey}),
consist of a first order differential equation on $\mathbb{R}$ as
the time-line 
\[
\partial_{0}v+Au=f,
\]
where $\partial_{0}$ denotes the derivative with respect to time
and $A$ is a suitable closed linear operator on a Hilbert space (frequently
a block-operator-matrix with spatial differential operators as entries).
The function $u$ and $v$ are the unknowns, while $f$ is a given
source term. This simple equation is completed by a so-called linear
material law linking $u$ and $v$:
\[
v=\mathcal{M}u,
\]
where $\mathcal{M}$ is an operator acting in time and space. Thus,
the differential equation becomes 
\begin{equation}
\left(\partial_{0}\mathcal{M}+A\right)u=f\label{eq:evo}
\end{equation}
a so-called evolutionary problem. The operator sum on the left-hand
side will be established in a suitable Hilbert space and thus, the
well-posedness of \prettyref{eq:evo} relies on the bounded invertibility
of that operator sum. For doing so, one establishes the time-derivative
$\partial_{0}$ as a normal boundedly invertible operator in a suitable
exponentially weighted $L_{2}-$space. With the spectral representation
of this normal operator at hand, one specifies the operator $\mathcal{M}$
as to be an analytic operator-valued function of $\partial_{0}^{-1}.$
Then $\mathcal{\ensuremath{M}}$ enjoys the property that it is causal
due to the Theorem of Paley and Wiener (see e.g. \cite[Theorem 19.2]{rudin1987real}).
Although the requirement of analyticity seems to be very strong, these
operators cover a broad class of possible space-time operators like
convolutions with suitable kernels naturally arising in the study
of integro-differential equations (see e.g. \cite{Trostorff2012_integro}),
translations with respect to time as they occur in delay equations
(see \cite{Kalauch2011}) as well as fractional derivatives (see \cite{Picard2013_fractional}).
Thus, the setting of evolutionary equations provides a unified framework
for a broad class of partial differential equations. We note that
the causality of $\mathcal{M}$ also yields the causality of the solution
operator $(\partial_{0}\mathcal{M}+A)^{-1}$ of our evolutionary problem
\prettyref{eq:evo}, which can be seen as a characterizing property
of evolutionary processes.\\
After establishing the well-posedness of \prettyref{eq:evo}, one
is interested in qualitative properties of the solution $u$. A first
property, which can be discussed, is the asymptotic behaviour of $u$,
especially the question of exponential stability. The study of stability
for differential equations goes back to Lyapunov and a lot of approaches
has been developed to tackle this question over the last decades.
We just like to mention some classical results for evolution equations,
using the framework of strongly continuous semigroups, like Datko's
Lemma \cite{Datko_1970} or the Theorem of Gearhart-Pr�ss \cite{Gearhart_1978,Pruss1984}
(see also \cite[Chapter V]{engel2000one} for the asymptotics of semigroups).
Unfortunately, these results are not applicable to evolutionary equations.
The main reason for that is that the solution $u$ of \prettyref{eq:evo}
is not continuous, unless the right-hand side $f$ is regular, so
that point-wise estimates for the solution $u$ (and this is how exponential
stability is usually defined) do not make any sense. Hence, we need
to introduce a more general notion of exponential stability for that
class of problems. This was done by the author in \cite{Trostorff2013_stability,Trostorff2014_PAMM}
(see also Subsection 2.2 in this article), where also sufficient conditions
on the material law $\mathcal{M}$ to obtain exponential stability
were derived.\\
The main purpose of this article is to study the exponential stability
of evolutionary problems of second order in time and space, i.e. to
equations of the form 
\begin{equation}
\left(\partial_{0}^{2}\mathcal{M}+C^{\ast}C\right)u=f,\label{eq:hyp}
\end{equation}
where $C$ is a densely defined closed linear operator, which is assumed
to be boundedly invertible. For doing so, we need to rewrite the above
problem as an evolutionary equation of first order in time. As it
turns out there are several ways to do this yielding a family of new
material law operators $\left(\mathcal{M}_{d}\right)_{d>0}$, such
that \prettyref{eq:hyp} can be written as 
\begin{equation}
\left(\partial_{0}\mathcal{M}_{d}+\left(\begin{array}{cc}
0 & C^{\ast}\\
-C & 0
\end{array}\right)\right)\left(\begin{array}{c}
\partial_{0}u+du\\
Cu
\end{array}\right)=\left(\begin{array}{c}
f\\
0
\end{array}\right),\label{eq:aux}
\end{equation}
for each $d>0$. The goal is now to state sufficient conditions for
the original operator $\mathcal{M},$ such that there is $d>0$ for
which the exponential stability of the equivalent problem \prettyref{eq:aux}
can be derived. This is the main objective of Subsection 3.1.\\
The article is structured as follows. Section 2 is devoted to the
framework of evolutionary equations. In this section we recall the
definition of the time-derivative $\partial_{0}$ and of linear material
laws. Moreover, we introduce the notion of exponential stability for
evolutionary equations and provide a characterization result using
Frequency Domain Methods (\prettyref{thm:characterization_exp_stability}).
As already indicated above, Subsection 3.1 deals with the exponential
stability of \prettyref{eq:hyp} or equivalently of \prettyref{eq:aux}
and provides sufficient constraints on the material law $\mathcal{M}$,
which yield the exponential stability. In Subsection 3.2 we focus
on special material law operators $\mathcal{M}$, namely convolutions
with a kernel $k$. We derive sufficient conditions on the kernel
$k$, such that the corresponding integro-differential equations becomes
exponentially stable. Exponential stability of hyperbolic integro-differential
equations was studied for instance in \cite{Cannarsa2011} and \cite{Pruss2009}
(where in \cite{Pruss2009} also polynomial stability was addressed).
We show that our conditions on the kernel $k$, which -- in contrast
to the aforementioned references -- is operator-valued, are weaker
than the ones imposed in \cite{Cannarsa2011,Pruss2009}. Finally,
in Section 4, we treat a concrete example of a wave equation with
convolution integral and time delay. This problem was recently studied
in \cite{Alabau2014} by using semigroup techniques and constructing
a suitable energy for which the exponential decay was shown. We will
see that this problem is covered by our abstract considerations and
hence, the exponential stability follows. Moreover, our approach allows
to relax the assumptions on the kernel of the convolution integral.

Throughout, all Hilbert spaces are assumed to be complex. Their inner
products are denoted by $\langle\cdot|\cdot\rangle,$ which are linear
in the second and conjugate linear in the first argument and the induced
norms are denoted by $|\cdot|$.

\section{Evolutionary problems}

In this section we recall the notion of evolutionary problems, as
they were first introduced in \cite{Picard}, \cite[Chapter 6]{Picard_McGhee}
(see also \cite{Picard2014_survey} for a survey). We begin by introducing
the time-derivative in an exponentially weighted $L_{2}$-space in
the first subsection. The second subsection is devoted to the well-posedness
and exponential stability of evolutionary problems. The main theorem
in this subsection (\prettyref{thm:characterization_exp_stability})
characterizes the exponential stability of an evolutionary problem
by point-wise properties of the unitarily equivalent multiplication
operators (such methods are frequently referred to as Frequency Domain
Methods). We remark that we generalize the notion of evolutionary
problems in the sense that we do not impose monotonicity constraints
on the operators involved. Throughout, let $H$ be a Hilbert space.

\subsection{The time-derivative}

We introduce the time-derivative as a boundedly invertible operator
on an exponentially weighted $L_{2}$-space. This idea first appears
in \cite{picard1989hilbert}. For the proofs of the forthcoming results,
we refer to \cite{Picard,Kalauch2011,Picard_McGhee}. For $\rho\in\mathbb{R}$
we consider the following Hilbert space 
\[
H_{\rho}(\mathbb{R};H)\coloneqq\left\{ f:\mathbb{R}\to H\,|\, f\mbox{ measurable,}\,\intop_{\mathbb{R}}|f(t)|^{2}e^{-2\rho t}\mbox{ d}t<\infty\right\} 
\]
equipped with the inner product 
\[
\langle f|g\rangle_{H_{\rho}(\mathbb{R};H)}\coloneqq\intop_{\mathbb{R}}\langle f(t)|g(t)\rangle_{H}e^{-2\rho t}\mbox{ d}t\quad(f,g\in H_{\rho}(\mathbb{R};H)).
\]
We note that for $\rho=0$ this is nothing but the usual $L_{2}(\mathbb{R};H).$
Moreover we define the operators 
\begin{align*}
e^{-\rho\mathrm{m}}:H_{\rho}(\mathbb{R};H) & \to L_{2}(\mathbb{R};H)\\
f & \mapsto\left(t\mapsto e^{-\rho t}f(t)\right)
\end{align*}
which are obviously unitary. We denote the derivative on $L_{2}(\mathbb{R};H)$
with maximal domain by $\partial$, i.e. 
\begin{align*}
\partial:H^{1}(\mathbb{R};H)\subseteq L_{2}(\mathbb{R};H) & \to L_{2}(\mathbb{R};H)\\
f & \mapsto f',
\end{align*}
where $H^{1}(\mathbb{R};H)$ is the classical Sobolev-space of $L_{2}$-functions
whose distributional derivative also belongs to $L_{2}.$ It is well-known
that this operator is skew-selfadjoint with $\sigma(\partial)=\i\mathbb{R}$
(see e.g. \cite[Example 3.14]{kato1995perturbation}). Moreover, it
is well-known that $\partial$ is unitarily equivalent to the multiplication
operator $\i\mathrm{m}$ on $L_{2}(\mathbb{R};H)$ with maximal domain,
i.e. 
\[
D(\mathrm{\i m})\coloneqq\left\{ f\in L_{2}(\mathbb{R};H)\,|\,\left(\i\mathrm{m}\right)f=(t\mapsto\i tf(t))\in L_{2}(\mathbb{R};H)\right\} ,
\]
where the unitary transformation is given by the Fourier-transform,
defined as the unitary extension of 
\[
\mathcal{F}|_{L_{1}(\mathbb{R};H)\cap L_{2}(\mathbb{R};H)}:L_{1}(\mathbb{R};H)\cap L_{2}(\mathbb{R};H)\subseteq L_{2}(\mathbb{R};H)\to L_{2}(\mathbb{R};H)
\]
with 
\[
\left(\mathcal{F}f\right)(t)\coloneqq\frac{1}{\sqrt{2\pi}}\intop_{\mathbb{R}}e^{-\i st}f(s)\mbox{ d}s\quad(f\in L_{1}(\mathbb{R};H)\cap L_{2}(\mathbb{R};H),\, t\in\mathbb{R}).
\]
In other words we have 
\[
\partial=\mathcal{F}^{\ast}\left(\i\mathrm{m}\right)\mathcal{F}.
\]
Using now the unitary operators $e^{-\rho\mathrm{m}}$ we define the
derivative $\partial_{0,\rho}$ on $H_{\rho}(\mathbb{R};H)$ by%
\footnote{\label{fn:domain}The equality of operators especially implies the
equality of their domains. Thus, the domain of $\partial_{0,\rho}$
is given by the natural domain of the operator $\left(e^{-\rho\mathrm{m}}\right)^{-1}\partial e^{-\rho\mathrm{m}}+\rho$,
which is
\[
\left\{ \left.u\in H_{\rho}(\mathbb{R};H)\,\right|\, e^{-\rho\mathrm{m}}u\in H^{1}(\mathbb{R};H)\right\} .
\]
} 
\[
\partial_{0,\rho}\coloneqq\left(e^{-\rho\mathrm{m}}\right)^{-1}\partial e^{-\rho\mathrm{m}}+\rho.
\]
 Indeed, this definition yields 
\[
\left(\partial_{0,\rho}\phi\right)(t)=e^{\rho t}\left(\phi'(t)e^{-\rho t}-\rho\phi(t)e^{-\rho t}\right)+\rho\phi(t)=\phi'(t)\quad(t\in\mathbb{R})
\]
for each $\phi\in C_{c}^{\infty}(\mathbb{R};H)$ -- the space of arbitrarily
differentiable functions on $\mathbb{R}$ with values in $H$ and
compact support. As an immediate consequence of this definition we
obtain that 
\[
\sigma(\partial_{0,\rho})=\rho+\i\mathbb{R},
\]
which in particular yields that $\partial_{0,\rho}$ is boundedly
invertible if and only if $\rho\ne0.$ Moreover, the spectrum is continuous
spectrum. The inverse $\partial_{0,\rho}^{-1}$ is given by 
\[
\left(\partial_{0,\rho}^{-1}f\right)(t)=\begin{cases}
\intop_{-\infty}^{t}f(s)\mbox{ d}s & \mbox{ if }\rho>0,\\
-\intop_{t}^{\infty}f(s)\mbox{ d}s & \mbox{ if }\rho<0.
\end{cases}
\]
Thus, $\rho>0$ corresponds to the so-called causal%
\footnote{Roughly speaking, causality means that the image at time $t$ just
depends on the values of the pre-image up to the same time $t$, while
anticausality means that it just depends on the values of the pre-image
beginning at time $t$. For a precise definition of causality, we
refer to \prettyref{rem:causality} (a).%
} case, while $\rho<0$ gives the anticausality of $\partial_{0,\rho}^{-1}$.
We further note that also $\partial_{0,\rho}$ is unitarily equivalent
to a multiplication operator on $L_{2}(\mathbb{R};H).$ The unitary
mapping is the so-called Fourier-Laplace transform given as 
\[
\mathcal{L}_{\rho}\coloneqq\mathcal{F}e^{-\rho\mathrm{m}}:H_{\rho}(\mathbb{R};H)\to L_{2}(\mathbb{R};H).
\]
Indeed, we obtain that 
\begin{align*}
\mathcal{L}_{\rho}\partial_{0,\rho}\mathcal{L}_{\rho}^{\ast} & =\mathcal{F}e^{-\rho\mathrm{m}}\left(\left(e^{-\rho\mathrm{m}}\right)^{-1}\partial e^{-\rho\mathrm{m}}+\rho\right)\left(e^{-\rho\mathrm{m}}\right)^{-1}\mathcal{F}^{\ast}\\
 & =\mathcal{F}\partial\mathcal{F}^{\ast}+\rho\\
 & =\i\mathrm{m}+\rho.
\end{align*}

\subsection{Well-posedness and exponential stability}

Throughout, let $A:D(A)\subseteq H\to H$ be a densely defined closed
linear operator. We define now, what we mean by a linear material
law.
\begin{defn*}
Let $\Omega\subseteq\mathbb{C}$ open such that%
\footnote{For $\mu\in\mathbb{R}$ we set 
\begin{align*}
\mathbb{C}_{\Re\gtreqqless\mu} & \coloneqq\left\{ z\in\mathbb{C}\,|\,\Re z\gtreqqless\mu\right\} ,\\
\mathbb{C}_{\Im\gtreqqless\mu} & \coloneqq\left\{ z\in\mathbb{C}\,|\,\Im z\gtreqqless\mu\right\} .
\end{align*}
} $\mathbb{C}_{\Re>\mu}\setminus\{0\}\subseteq[\Omega]^{-1}=\{z^{-1}\,|\, z\in\Omega\}$
for some $\mu\in\mathbb{R}.$ A \emph{linear material law} is an analytic
mapping $M:\Omega\to L(H).$ 
\end{defn*}
As multiplication operators will play an important role in the framework
of evolutionary problems, we introduce them in the following definition.
\begin{defn*}
For $t\in\mathbb{R}$ let $T(t)$ be a linear operator on $H$. Then
we set 
\begin{align*}
T(\m):D(T(\m))\subseteq L_{2}(\mathbb{R};H) & \to L_{2}(\mathbb{R};H)\\
f & \mapsto\left(t\mapsto T(t)f(t)\right),
\end{align*}
where 
\[
D(T(\m))\coloneqq\left\{ f\in L_{2}(\mathbb{R};H)\,|\, f(t)\in D(T(t))\mbox{ for a.e. }t\in\mathbb{R},\,(t\mapsto T(t)f(t))\in L_{2}(\mathbb{R};H)\right\} .
\]

\end{defn*}
With the help of this notion we are able to define so-called evolutionary
problems.
\begin{defn*}
Let $M$ be a linear material law. We associate the multiplication
operator $M\left(\frac{1}{\i\m+\rho}\right)$ on $L_{2}(\mathbb{R};H)$
for $\rho$ large enough, i.e. 
\[
\left(M\left(\frac{1}{\i\mathrm{m}+\rho}\right)f\right)(t)\coloneqq M\left(\frac{1}{\i t+\rho}\right)f(t),
\]
where $f\in L_{2}(\mathbb{R};H)$ such that $\left(t\mapsto M\left(\frac{1}{\i t+\rho}\right)f(t)\right)\in L_{2}(\mathbb{R};H)$.
Moreover, we consider its unitarily equivalent operator 
\[
M(\partial_{0,\rho}^{-1})\coloneqq\mathcal{L}_{\rho}^{\ast}M\left(\frac{1}{\i\m+\rho}\right)\mathcal{L}_{\rho}
\]
with its natural domain (cp. Footnote \ref{fn:domain}). An \emph{evolutionary
problem} is an equation in $H_{\rho}(\mathbb{R};H)$ of the form 
\begin{equation}
\left(\partial_{0,\rho}M(\partial_{0,\rho}^{-1})+A\right)u=f,\label{eq:evo1}
\end{equation}
where we identify $A$ with its canonical extension to $H_{\rho}(\mathbb{R};H)$
given by $\left(Au\right)(t)\coloneqq A\left(u(t)\right)$ for almost
every $t\in\mathbb{R}$ and $u\in H_{\rho}(\mathbb{R};H)$ such that
$u(t)\in D(A)$ for almost every $t\in\mathbb{R}$ and $\left(t\mapsto A\left(u(t)\right)\right)\in H_{\rho}(\mathbb{R};H).$
\end{defn*}
Let us illustrate the class of evolutionary equations by some examples.
\begin{example}
$\,$

\begin{enumerate}[(a)]

\item Let $A=\left(\begin{array}{cc}
0 & -C^{\ast}\\
C & 0
\end{array}\right):D(C)\times D(C^{\ast})\subseteq H_{0}\oplus H_{1}\to H_{0}\oplus H_{1}$, where $C:D(C)\subseteq H_{0}\to H_{1}$ is a densely defined closed
linear operator%
\footnote{A typical example could be $C=\grad,$ where $D(C)=H_{0}^{1}(\Omega),H_{0}=L_{2}(\Omega)$
and $H_{1}=L_{2}(\Omega)^{n}.$ Then $C^{\ast}=-\dive$, the divergence
on $L_{2}$. But also $C=\curl$ is possible, which allows the treatment
of Maxwell's equations. %
} between two Hilbert spaces $H_{0},H_{1}$. By setting $M(z)\coloneqq M_{0}+zM_{1}$
for operators $M_{0},M_{1}\in L(H)$ with $H\coloneqq H_{0}\oplus H_{1}$
we cover a class of parabolic, hyperbolic and elliptic problems. Indeed,
if $M_{0}=\left(\begin{array}{cc}
M_{00} & 0\\
0 & 0
\end{array}\right),$$M_{1}=\left(\begin{array}{cc}
0 & 0\\
0 & M_{11}
\end{array}\right)$ the problem reads as 
\[
\left(\partial_{0,\rho}\left(M_{0}+\partial_{0,\rho}^{-1}M_{1}\right)+A\right)u=f.
\]
In particular, letting $f=\left(\begin{array}{c}
h\\
0
\end{array}\right)\in H_{\rho,0}(\mathbb{R};H)$ and setting $u=\left(\begin{array}{c}
u_{0}\\
u_{1}
\end{array}\right)$ we read off that 
\begin{align*}
\partial_{0,\rho}M_{00}u_{0}-C^{\ast}u_{1} & =f,\\
M_{11}u_{1}+Cu_{0} & =0.
\end{align*}
Thus, assuming that $M_{11}$ is boundedly invertible, the second
equation reads as $u_{1}=-M_{11}^{-1}Cu_{0}$ and thus, we indeed
end up with an equation of parabolic type for $u_{0}$ 
\[
\partial_{0,\rho}M_{00}u_{0}+C^{\ast}M_{11}^{-1}Cu_{0}=f.
\]

If $M_{0}=\left(\begin{array}{cc}
M_{00} & 0\\
0 & M_{01}
\end{array}\right)$ and $M_{1}=0$ we get 
\begin{align*}
\partial_{0,\rho}M_{00}u_{0}-C^{\ast}u_{1} & =f,\\
\partial_{0,\rho}M_{01}u_{1}+Cu_{0} & =0.
\end{align*}
Again, assuming that $M_{01}$ is boundedly invertible, we obtain
$u_{1}=-\partial_{0,\rho}^{-1}M_{01}^{-1}Cu_{0}$ and thus, the first
equation reads as 
\[
\partial_{0,\rho}M_{00}u_{0}+C^{\ast}\partial_{0,\rho}^{-1}M_{01}^{-1}Cu_{0}=f.
\]
Differentiation yields 
\[
\partial_{0,\rho}^{2}M_{00}u_{0}+C^{\ast}M_{01}^{-1}Cu_{0}=\partial_{0,\rho}f,
\]
which gives an equation of hyperbolic type. Finally, choosing $M_{0}=0$
and $M_{1}=\left(\begin{array}{cc}
M_{10} & 0\\
0 & M_{11}
\end{array}\right)$ we end up with an elliptic type problem of the form 
\begin{align*}
M_{10}u_{0}-C^{\ast}u_{1} & =f,\\
M_{11}u_{1}+Cu_{0} & =0,
\end{align*}

which my be rewritten as 
\[
M_{10}u_{0}+C^{\ast}M_{11}^{-1}Cu_{0}=f.
\]

Also, problems of mixed type are treatable, i.e. equations which are
hyperbolic on one part of the domain, parabolic on another one and
elliptic on a third part (see e.g. \cite[Example 2.43]{Picard2014_survey},
\cite[p.8]{Waurick2013_continuous_dep}). It should be noted that
in all previous examples, $M_{0}$ and $M_{1}$ are also allowed to
have non-vanishing off-diagonal entries. Moreover, in the abstract
setting of evolutionary equations, there is no need to assume that
the block structures of $A$ and $M_{0}$ and $M_{1}$ are comparable.
This allows the treatment of even more general differential equations. 

\item Let $H=L_{2}(\Omega)$ for some $\Omega\subseteq\mathbb{R}^{n}$
and $k:\mathbb{R}_{\geq0}\to\mathbb{R}$ be a measurable integrable
function. Setting $M(z)=\sqrt{2\pi}\hat{k}(-\i z^{-1})$, where $\hat{k}$
denotes the Fourier transform of $k$, we end up with an evolutionary
problem of the form 
\[
\partial_{0,\rho}k\ast u+Au=f,
\]
which is an integro-differential equation. For a detailed study of
integro-differential equations within the framework of evolutionary
problems we refer to \cite{Trostorff2012_integro}. A concrete example
is also treated in Section 4.

\item Setting $M(z)=z^{-\alpha}$ for some $\alpha\in]0,1[$ we get
\[
\partial_{0,\rho}\partial_{0,\rho}^{-\alpha}u+Au=\partial_{0,\rho}^{1-\alpha}u+Au=f,
\]
which covers a class of fractional differential equations. For more
details and a more complicated examples we refer to \cite{Picard2013_fractional}.

\end{enumerate}\end{example}
\begin{lem}
\label{lem:mult_operator}Let $\rho\in\mathbb{R}$. We consider a
continuous mapping $L:\left\{ \i t+\rho\,|\, t\in\mathbb{R}\right\} \to L(H)$.
For $t\in\mathbb{R}$ we define 
\begin{align*}
T(\i t+\rho):D(A)\subseteq H & \to H\\
x & \mapsto\left(L(\i t+\rho)+A\right)x.
\end{align*}
We consider the corresponding multiplication operators $T(\i\m+\rho)$
and $L(\i\m+\rho)$ on $L_{2}(\mathbb{R};H)$. Then%
\footnote{Note that the natural domain of the operator $L(\i\mathrm{m}+\rho)+A$
is given by 
\begin{align*}
 & D(L(\i\mathrm{m}+\rho))\cap D(A)\\
 & =\left\{ \left.f\in L_{2}(\mathbb{R};H)\,\right|\, f(t)\in D(A)\mbox{ for a.e. }t\in\mathbb{R},\,\left(t\mapsto L(\i t+\rho)f(t)\right)\in L_{2}(\mathbb{R};H),\,\left(t\mapsto Af(t)\right)\in L_{2}(\mathbb{R};H)\right\} ,
\end{align*}
which is in general a proper subset of 
\[
D(T(\i\mathrm{m}+\rho))=\left\{ f\in L_{2}(\mathbb{R};H)\,|\, f(t)\in D(A)\mbox{ for a.e. }t\in\mathbb{R},\,\left(t\mapsto L(\i t+\rho)f(t)+Af(t)\right)\in L_{2}(\mathbb{R};H)\right\} .
\]
} 
\[
T(\i\m+\rho)=\overline{L(\i\m+\rho)+A},
\]
where we identify $A$ with its canonical extension to $L_{2}(\mathbb{R};H)$.\end{lem}
\begin{proof}
We first prove that $T(\i\m+\rho)$ is closed. For doing so, let $(f_{n})_{n\in\mathbb{N}}$
in $D(T(\i\m+\rho))$ with $f_{n}\to f$ and $T(\i\m+\rho)f_{n}\to g$
in $L_{2}(\mathbb{R};H).$ By passing to a suitable subsequence, we
may assume without loss of generality, that $f_{n}(t)\to f(t)$ and
$(T(\i\m+\rho)f_{n})(t)\to g(t)$ for almost every $t\in\mathbb{R}.$
As $L(\i t+\rho)\in L(H)$ for each $t\in\mathbb{R},$ we derive that
\[
L(\i t+\rho)f_{n}(t)\to L(\i t+\rho)f(t)
\]
for almost every $t\in\mathbb{R}.$ Consequently 
\[
Af_{n}(t)=T(\i t+\rho)f_{n}(t)-L(\i t+\rho)f_{n}(t)\to g(t)-L(\i t+\rho)f(t)
\]
for almost every $t\in\mathbb{R}$ and hence, by the closedness of
$A$ we obtain $f(t)\in D(A)$ and $L(\i t+\rho)f(t)+Af(t)=g(t)$
almost everywhere. This shows $f\in D(T(\i\m+\rho))$ and $T(\i\m+\rho)f=g$
and thus, $T(\i\m+\rho)$ is closed. Since trivially 
\[
L(\i\m+\rho)+A\subseteq T(\i\m+\rho)
\]
 we derive 
\[
\overline{L(\i\m+\rho)+A}\subseteq T(\i\m+\rho).
\]
For showing the converse inclusion, let $f\in D(T(\i\m+\rho)).$ For
$n\in\mathbb{N}$ we define $f_{n}\coloneqq\chi_{[-n,n]}(\m)f$ by
$f_{n}(x)\coloneqq\chi_{[-n,n]}(x)f(x)$ for $x\in\mathbb{R}.$ We
estimate 
\[
\intop_{\mathbb{R}}\left|L(\i t+\rho)f_{n}(t)\right|^{2}\mbox{ d}t=\intop_{-n}^{n}\left|L(\i t+\rho)f(t)\right|^{2}\mbox{ d}t\leq\sup_{t\in[-n,n]}\|L(\i t+\rho)\|^{2}|f|_{L_{2}(\mathbb{R};H)}^{2},
\]
showing that $f_{n}\in D(L(\i\m+\rho))$. Since clearly $f_{n}\in D(T(\i\m+\rho))$
we obtain that $f_{n}\in D\left(L(\i\m+\rho)+A\right)$. Now, since
$f_{n}\to f$ as well as $\left(L(\i\m+\rho)+A\right)f_{n}=T(\i\m+\rho)f_{n}\to T(\i\m+\rho)f$
in $L_{2}(\mathbb{R};H)$ by dominated convergence, we conclude $f\in D\left(\overline{L(\i\m+\rho)+A}\right)$.
\end{proof}
The latter lemma implies that the operator $\overline{\partial_{0,\rho}M(\partial_{0,\rho}^{-1})+A}$
is unitarily equivalent (via the Fourier-Laplace transformation $\mathcal{L}_{\rho}$)
to the multiplication operator $T(\i\m+\rho),$ where $T(z)=zM(z^{-1})+A.$
Indeed, for $L(z)\coloneqq zM(z^{-1})$ we get that 
\begin{equation}
\mathcal{L}_{\rho}\left(\overline{\partial_{0,\rho}M(\partial_{0,\rho}^{-1})+A}\right)\mathcal{L}_{\rho}^{\ast}=\overline{\mathcal{L}_{\rho}\partial_{0,\rho}M(\partial_{0,\rho}^{-1})\mathcal{L}_{\rho}^{\ast}+A}=\overline{L(\i\m+\rho)+A}=T(\i\m+\rho).\label{eq:unitary_eq}
\end{equation}
Hence, the study of evolutionary equations is equivalent to the study
of multiplication operators. This observation allows us to define
well-posedness of an evolutionary problem in terms of the function
$T$. 
\begin{defn*}
Let $M$ be a linear material law. We call the associated evolutionary
problem \emph{
\[
\left(\partial_{0,\rho}M(\partial_{0,\rho}^{-1})+A\right)u=f
\]
well-posed, }if there exists $\rho_{0}\in\mathbb{R}$ such that for
each $z\in\mathbb{C}_{\Re>\rho_{0}}\setminus\{0\}$ the operator $zM(z^{-1})+A$
is boundedly invertible with $\sup_{z\in\mathbb{C}_{\Re>\rho_{0}}\setminus\{0\}}\left\Vert \left(zM(z^{-1})+A\right)^{-1}\right\Vert <\infty.$
The infimum over all such numbers $\rho_{0}$ is called the growth
bound of the evolutionary problem and we denote it by $\omega_{0}(M,A).$ \end{defn*}
\begin{rem}
\label{rem:causality}$\,$

\begin{enumerate}[(a)]

\item We note that for a well-posed evolutionary problem we have
that 
\[
\overline{\partial_{0,\rho}M(\partial_{0,\rho}^{-1})+A}
\]
is boundedly invertible for each $\rho>\rho_{0}$ due to \prettyref{eq:unitary_eq}.
Moreover, by a Paley-Wiener-type argument (see e.g. \cite[Theorem 19.2]{rudin1987real}),
we obtain that the solution operator $S_{\rho}\coloneqq\left(\overline{\partial_{0,\rho}M(\partial_{0,\rho}^{-1})+A}\right)^{-1}$
is forward causal, in the sense that 
\[
\chi_{]-\infty,a]}(\m)S_{\rho}=\chi_{]-\infty,a]}(\m)S_{\rho}\chi_{]-\infty,a]}(\m).
\]
For more details we refer to \cite{Picard,Picard_McGhee}. 

\item For the evolutionary problems originally treated in \cite{Picard},
the operator $A$ was assumed to be skew-selfadjoint, while the material
law%
\footnote{For $z_{0}\in\mathbb{C}$ and $r>0$ we define $B(z_{0},r)\coloneqq\{z\in\mathbb{C}\,|\,|z-z_{0}|<r\}$.
Likewise $B[z_{0},r]\coloneqq\{z\in\mathbb{C}\,|\,|z-z_{0}|\leq r\}=\overline{B(z_{0},r)}.$%
} $M:B(r,r)\to L(H)$ was assumed to satisfy a positive definiteness
constraint of the form%
\footnote{For an operator $T\in L(H)$ we denote by $\Re T$ the selfadjoint
operator $\frac{1}{2}(T+T^{\ast}).$%
} 
\[
\forall z\in B(r,r):\Re z^{-1}M(z)\geq c>0.
\]
Clearly, these assumptions yield the well-posedness of the evolutionary
problem in the above sense. 

\item Evolutionary problems may be written as an abstract operator
equation of the form $\left(B+A\right)u=f,$ where $B\coloneqq\partial_{0,\rho}M(\partial_{0,\rho}^{-1})$.
Such abstract problems were studied in a Banach space setting by da
Prato and Grisvard \cite{daPrato1975} and the results were applied
to differential equations of hyperbolic and parabolic type. However,
these results are not applicable here, since the desired spectral
properties of the operators involved are not met. Indeed, in \cite{daPrato1975}
the operators have to verify, besides other spectral properties, a
condition of Hille-Yosida type, which needs not to be satisfied for
our choice of $A$ and $B$. \\
In \cite[Chapter 5]{Favini1999} Favini and Yagi studied so-called
degenerated differential equations on Banach spaces, which are of
the form $\left(TM-L\right)u=f,$ where $T,M$ and $L$ are operators
on some Banach space. Moreover, the resolvent sets of $T$ and $M$
should contain certain logarithmic regions, while $L$ is boundedly
invertible and certain compatibility conditions for $T,M$ and $L$
are required. In our case this would correspond to $T=\partial_{0,\rho},$
$M=M(\partial_{0,\rho}^{-1})$ and $L=-A,$ which in general do not
satisfy these assumptions. 

\end{enumerate}
\end{rem}
Moreover, we obtain that the solution operator $S_{\rho}$ does not
depend on the particular choice of the parameter $\rho$ as the following
proposition shows. Thus, we usually will omit the index $\rho$ in
$\partial_{0,\rho}$ and just write $\partial_{0}$ instead.
\begin{prop}[{\cite[Lemma 3.6]{Trostorff2013_stability}}]
\label{prop:independence} Let $\rho,\rho'\in\mathbb{R}$ with $\rho'>\rho$
and set $U\coloneqq\{z\in\mathbb{C}\,|\,\Re z\in[\rho,\rho']\}$.
Moreover, let $S:U\to L(H)$ be continuous, bounded and analytic on
$\interior U$ and $f\in H_{\rho}(\mathbb{R};H)\cap H_{\rho'}(\mathbb{R};H)$.
Then 
\[
\left(\mathcal{L}_{\rho}^{\ast}S(\i\m+\rho)\mathcal{L}_{\rho}f\right)(t)=\left(\mathcal{L}_{\rho'}^{\ast}S(\i\m+\rho')\mathcal{L}_{\rho'}f\right)(t)
\]
for almost every $t\in\mathbb{R}.$
\end{prop}
We are now ready to introduce the notion of an exponentially stable
evolutionary problem as it was defined in \cite{Trostorff2013_stability,Trostorff2014_PAMM}.
\begin{defn*}
A well-posed evolutionary problem is called \emph{exponentially stable}
\emph{with stability rate $\rho_{1}>0,$} if for each $0\leq\nu<\rho_{1}$
and $\rho>\omega_{0}(M,A)$ and $f\in H_{-\nu}(\mathbb{R};H)\cap H_{\rho}(\mathbb{R};H)$
we have that 
\[
\left(\overline{\partial_{0}M(\partial_{0}^{-1})+A}\right)^{-1}f\in H_{-\nu}(\mathbb{R};H).
\]
\end{defn*}
\begin{rem}
$\,$

\begin{enumerate}[(a)]

\item We note that we cannot use the standard notion of exponential
stability as it is used for instance in semigroup theory. There, one
usually requires point-wise estimates for the solution $u$ of the
form $|u(t)|\leq Me^{-\omega t}$ for some $M,\omega>0$ and each
$t\in\mathbb{R}_{\geq0}.$ The main problem is that we do not have
any regularizing property of our solution operator $\left(\overline{\partial_{0}M(\partial_{0}^{-1})+A}\right)^{-1}$
allowing to get continuous solutions. Thus, pointwise estimates cannot
be used in our framework. Indeed, choosing for instance $M(\partial_{0}^{-1})\coloneqq\partial_{0}^{-1}$
we end up with a purely algebraic equation given by 
\[
\left(1+A\right)u=f,
\]
where we cannot expect continuity of the solution $u$, unless our
right-hand side $f$ is more regular than just square integrable.
However, the notion of exponential stability as introduced above yields
a classical point-wise estimate of the solution, if the right-hand
side $f$ is regular enough, for example an element in the domain
of $\partial_{0}$ (see \cite[Remark 3.2 (a)]{Trostorff2013_stability}).

\item We further remark that exponential stability is not just the
requirement that we can solve the evolutionary problem for negative
$\rho$. The main problem is, that we need to preserve the causality
of the solution operator, which is a typical property for positive
but not for negative $\rho$ (compare $\partial_{0,\rho}^{-1}$ in
dependence of $\rho$). That is why we need to define the exponential
stability in terms of the causal solution operator, which is guaranteed
by the choice $\rho>\omega_{0}(M,A)$ in the latter definition. Indeed,
consider the simple case $A=0$ and $M(\partial_{0}^{-1})=1.$ Then
the corresponding evolutionary problem reads as 
\[
\partial_{0,\rho}u=f
\]
and we have $\omega(M,A)=\omega(1,0)=0.$ This problem is solvable
for positive and negative $\rho$, yielding however, different solutions.
Choosing for instance $f=\chi_{[0,1]}$ we get 
\begin{align*}
u(t) & =\begin{cases}
\intop_{-\infty}^{t}f(s)\mbox{ d}s & \mbox{ if }\rho>0,\\
-\intop_{t}^{\infty}f(s)\mbox{ d}s & \mbox{ if }\rho<0
\end{cases}\\
 & =\begin{cases}
t\chi_{[0,1]}(t)+\chi_{]1,\infty[}(t) & \mbox{ if }\rho>0,\\
-\left(\chi_{-]-\infty,0[}(t)+(1-t)\chi_{[0,1]}(t)\right) & \mbox{ if }\rho<0
\end{cases}
\end{align*}
for $t\in\mathbb{R}$, which shows that even if $f$ decays exponentially
(it is even compactly supported) the causal solution (corresponding
to positive $\rho$) is not exponentially decaying. 

\end{enumerate}
\end{rem}
In the subsequent theorem we will give a characterization of exponential
stability in terms of the resolvents of $zM(z^{-1})+A$ for $z\in\mathbb{C}_{\Re>-\rho_{1}}\setminus\{0\}.$
For doing so, we need the following auxiliary result. 
\begin{thm}[\cite{Weiss1991}]
 \label{thm:H_infty rep}Let $S:L_{2}(\mathbb{R}_{\geq0};H)\to L_{2}(\mathbb{R}_{\geq0};H)$
be a bounded, shift-invariant linear operator. Then there exists a
uniquely determined bounded and analytic function $N:\mathbb{C}_{\Re>0}\to L(H)$
such that 
\[
\left(\mathcal{L}_{\Re z}Sf\right)(\Im z)=N(z)\left(\mathcal{L}_{\Re z}f\right)(\Im z)
\]
for each $f\in L_{2}(\mathbb{R}_{\geq0};H)$ and every $z\in\mathbb{C}_{\Re>0}$.
\end{thm}
Now we are ready to state our characterization result.
\begin{thm}
\label{thm:characterization_exp_stability}Let $M:\mathbb{C}\setminus B[-r,r]\to L(H)$
be analytic and $0<\rho_{1}<\frac{1}{2r}$. We assume that the evolutionary
problem 
\[
\left(\partial_{0}M(\partial_{0}^{-1})+A\right)u=f
\]
is well-posed. Then the following statements are equivalent:

\begin{enumerate}[(i)]

\item The evolutionary problem $\left(\partial_{0}M(\partial_{0}^{-1})+A\right)u=f$
is exponentially stable with stability rate $\rho_{1},$

\item For each $z\in\mathbb{C}_{\Re>-\rho_{1}}\setminus\{0\}$ we
have $0\in\rho\left(zM(z^{-1})+A\right)$ and the function 
\begin{align*}
\mathbb{C}_{\Re>-\rho_{1}}\setminus\{0\} & \to L(H)\\
z & \mapsto\left(zM(z^{-1})+A\right)^{-1}
\end{align*}
is bounded.

\end{enumerate}\end{thm}
\begin{proof}
Since our evolutionary problem is assumed to be well-posed, there
is $\rho_{0}\in\mathbb{R}$ such that 
\[
\mathbb{C}_{\Re>\rho_{0}}\setminus\{0\}\ni z\mapsto(zM(z^{-1})+A)^{-1}\in L(H)
\]
is bounded and analytic.\\
(i)$\Rightarrow$(ii): The proof is done in 3 steps.\\
Step 1: We show that the operator 
\[
S\coloneqq e^{\rho_{1}\m}\left(\overline{\partial_{0}M(\partial_{0}^{-1})+A}\right)^{-1}e^{-\rho_{1}\m}:L_{2}(\mathbb{R}_{\geq0};H)\to L_{2}(\mathbb{R}_{\geq0};H)
\]
satisfies the assumptions of \prettyref{thm:H_infty rep} and thus,
there is $N:\mathbb{C}_{\Re>0}\to L(H)$ analytic and bounded such
that 
\[
\left(\mathcal{L}_{\Re z}Sf\right)(\Im z)=N(z)\left(\mathcal{L}_{\Re z}f\right)(\Im z)
\]
for each $f\in L_{2}(\mathbb{R}_{\geq0};H)$ and every $z\in\mathbb{C}_{\Re>0}$.
\\
Indeed, $S$ is well-defined since for $f\in L_{2}(\mathbb{R}_{\geq0};H)$
it follows that $e^{-\rho_{1}\m}f\in\bigcap_{\rho\geq-\rho_{1}}H_{\rho}(\mathbb{R}_{\geq0};H),$
and thus, by assumption and the causality of $\left(\overline{\partial_{0}M(\partial_{0}^{-1})+A}\right)^{-1},$
we obtain that $Sf\in L_{2}(\mathbb{R}_{\geq0};H).$ Moreover $S$
is closed. Indeed, let $(f_{n})_{n\in\mathbb{N}}$ be a sequence in
$L_{2}(\mathbb{R}_{\geq0};H)$ such that $f_{n}\to f$ and $Sf_{n}\to g$
in $L_{2}(\mathbb{R}_{\geq0};H)$ for some $f,g\in L_{2}(\mathbb{R}_{\geq0};H)$.
Consequently $e^{-\rho_{1}\m}f_{n}\to e^{-\rho_{1}\m}f$ in $H_{\rho}(\mathbb{R}_{\geq0};H)$
for each $\rho\geq-\rho_{1}.$ By the boundedness of $\left(\overline{\partial_{0}M(\partial_{0}^{-1})+A}\right)^{-1}$
on $H_{\rho}(\mathbb{R}_{\geq0};H)$ for $\rho>\rho_{0}$ we derive
that 
\[
\left(\overline{\partial_{0}M(\partial_{0}^{-1})+A}\right)^{-1}e^{-\rho_{1}\m}f_{n}\to\left(\overline{\partial_{0}M(\partial_{0}^{-1})+A}\right)^{-1}e^{-\rho_{1}\m}f
\]
in $H_{\rho}(\mathbb{R}_{\geq0};H)$ and hence, $Sf_{n}\to e^{\rho_{1}\m}\left(\overline{\partial_{0}M(\partial_{0}^{-1})+A}\right)^{-1}e^{-\rho_{1}\m}f=Sf$
in $H_{\rho+\rho_{1}}(\mathbb{R}_{\geq0};H)$. As $L_{2}(\mathbb{R}_{\geq0};H)\hookrightarrow H_{\rho}(\mathbb{R}_{\geq0};H)$
for each $\rho\geq0,$ we derive that $g=Sf$ and thus, $S$ is closed.
Hence, according to the closed graph theorem, we get that $S$ is
bounded. Since $S$ is obviously translation invariant, by \prettyref{thm:H_infty rep}
there exists a unique analytic and bounded function $N:\mathbb{C}_{\Re>0}\to L(H)$
such that 
\[
\left(\mathcal{L}_{\Re z}Sf\right)(\Im z)=N(z)\left(\mathcal{L}_{\Re z}f\right)(\Im z)
\]
for each $f\in L_{2}(\mathbb{R}_{\geq0};H)$ and every $z\in\mathbb{C}_{\Re>0}.$
\\
Step 2: We show that $N(z+\rho_{1})=\left(zM\left(z^{-1}\right)+A\right)^{-1}$
for each $z\in\mathbb{C}_{\Re>\rho_{0}}.$ \\
Since for $\rho>\rho_{0}+\rho_{1}$ we have that (cp. \prettyref{eq:unitary_eq})
\begin{align*}
\left(\mathcal{L}_{\rho}Sf\right)(t) & =\left(\mathcal{L}_{\rho}e^{\rho_{1}\m}\left(\overline{\partial_{0}M(\partial_{0}^{-1})+A}\right)^{-1}e^{-\rho_{1}\m}f\right)(t)\\
 & =\left(\mathcal{L}_{\rho-\rho_{1}}\left(\overline{\partial_{0}M(\partial_{0}^{-1})+A}\right)^{-1}e^{-\rho_{1}\m}f\right)(t)\\
 & =\left((\i t+\rho-\rho_{1})M\left(\frac{1}{\i t+\rho-\rho_{1}}\right)+A\right)^{-1}\left(\mathcal{L}_{\rho-\rho_{1}}e^{-\rho_{1}\m}f\right)(t),\\
 & =\left((\i t+\rho-\rho_{1})M\left(\frac{1}{\i t+\rho-\rho_{1}}\right)+A\right)^{-1}\left(\mathcal{L}_{\rho}f\right)(t)
\end{align*}
for every $f\in L_{2}(\mathbb{R}_{\geq0};H)$, we derive that 
\[
N(z+\rho_{1})=\left(zM\left(z^{-1}\right)+A\right)^{-1}
\]
for each $z\in\mathbb{C}_{\Re>\rho_{0}}.$ 

Step 3: We show (ii), i.e. $0\in\rho(zM(z^{-1})+A)$ for each $z\in\mathbb{C}_{\Re>-\rho_{1}}\setminus\{0\}.$
\\
We consider 
\[
\Omega\coloneqq\left\{ \left.z\in\mathbb{C}_{\Re>-\rho_{1}}\setminus\{0\}\,\right|\,0\in\rho(zM(z^{-1})+A)\right\} \subseteq\mathbb{C}_{\Re>-\rho_{1}}\setminus\{0\}.
\]
We first prove that $\Omega$ is open. For doing so let $z'\in\Omega$.
Then for $z\in\mathbb{C}_{\Re>-\rho_{1}}\setminus\{0\}$ we compute
\begin{align}
zM(z^{-1})+A & =zM(z{}^{-1})-z'M(z'^{-1})+z'M(z'^{-1})+A\nonumber \\
 & =\left(\left(zM(z^{-1})-z'M(z'^{-1})\right)\left(z'M(z'^{-1})+A\right)^{-1}+1\right)\left(z'M(z'^{-1})+A\right).\label{eq:Neumann}
\end{align}
Due to the continuity of $M$ there exists $\delta>0$ such that for
$|z-z'|<\delta$ we have that 
\[
\left\Vert \left(zM(z^{-1})-z'M(z'^{-1})\right)\left(z'M(z'^{-1})+A\right)^{-1}\right\Vert <1.
\]
Hence, by the Neumann series, $0\in\rho(zM(z^{-1})+A)$ for each $|z-z'|<\delta,$
showing that $\Omega$ is open. \\
We consider now the component $C$ of $\rho_{0}+1\in\Omega$ in $\Omega.$
This component is open, since $\Omega$ is open. Moreover, due to
the identity theorem, we obtain that $N(z+\rho_{1})=\left(zM\left(z^{-1}\right)+A\right)^{-1}$
for each $z\in C.$ We will show that $C$ is also closed in $\mathbb{C}_{\Re>-\rho_{1}}\setminus\{0\}$.
For doing so, let $(z_{n})_{n\in\mathbb{N}}$ be a sequence in $C$
converging to some $z\in\mathbb{C}_{\Re>-\rho_{1}}\setminus\{0\}$.
Since 
\[
\sup_{n\in\mathbb{N}}\left\Vert \left(z_{n}M\left(z_{n}^{-1}\right)+A\right)^{-1}\right\Vert =\sup_{n\in\mathbb{N}}\|N(z_{n}+\rho_{1})\|<\infty,
\]
we obtain by \prettyref{eq:Neumann} (replace $z'$ by $z_{n}$) that
$z\in\Omega.$ Since $\Omega$ is open we find $\varepsilon>0$ with
$B(z,\varepsilon)\subseteq\Omega.$ Moreover, there is $n\in\mathbb{N}$
with $z_{n}\in B(z,\varepsilon)$ showing that $z$ and $z_{n}$ belong
to the same component $C$. Hence, $C$ is clopen in $\mathbb{C}_{\Re>-\rho_{1}}\setminus\{0\}$
and thus, $C=\mathbb{C}_{\Re>-\rho_{1}}\setminus\{0\}.$ The latter
gives $\Omega=\mathbb{C}_{\Re>-\rho_{1}}\setminus\{0\}$ and by the
identity theorem we conclude 
\[
(zM(z^{-1})+A)^{-1}=N(z+\rho_{1})\in L(H)
\]
for each $z\in\mathbb{C}_{\Re>-\rho_{1}}\setminus\{0\}$.\\
(ii)$\Rightarrow$(i): By assumption, the function 
\begin{align*}
S:\mathbb{C}_{\Re>-\rho_{1}}\setminus\{0\} & \to L(H)\\
z & \mapsto\left(zM(z^{-1})+A\right)^{-1}
\end{align*}
is analytic and bounded and hence, its singularity in $0$ is removable.
We denote its analytic extension to $\mathbb{C}_{\Re>-\rho_{1}}$
again by $S$. Consequently, for each $\rho>-\rho_{0}$ the multiplication
operator $S(\i\m+\rho)$ is bounded. Its inverse is given by the multiplication
$T(\i\m+\rho)$, where $T(z)\coloneqq(zM(z^{-1})+A).$ Hence, using
\prettyref{lem:mult_operator}, we obtain for each $\rho>-\rho_{1}$
\[
S(\i\m+\rho)=\left(\overline{\left(\i\m+\rho\right)M\left(\frac{1}{\i\m+\rho}\right)+A}\right)^{-1}.
\]
Let now $f\in H_{-\nu}(\mathbb{R};H)\cap H_{\rho}(\mathbb{R};H)$
for some $0\leq\nu<-\rho_{1}$ and $\rho>\omega_{0}(M,A).$ Then,
by \prettyref{prop:independence} 
\[
\left(\overline{\partial_{0,\rho}M(\partial_{0,\rho}^{-1})+A}\right)^{-1}f=\mathcal{L}_{\rho}^{\ast}S(\i\m+\rho)\mathcal{L}_{\rho}f=\mathcal{L}_{-\nu}^{\ast}S(\i\m-\nu)\mathcal{L}_{-\nu}f\in H_{-\nu}(\mathbb{R};H).\tag*{\qedhere}
\]

\end{proof}
As immediate consequences of the latter theorem we obtain the following
two corollaries.
\begin{cor}
Let $M:\mathbb{C}\setminus B[-r,r]\to L(H)$ be analytic and assume
that the corresponding evolutionary problem \prettyref{eq:evo1} is
well-posed. Then it is exponentially stable with stability rate $0<\rho_{1}<\frac{1}{2r}$
if and only if $\omega_{0}(M,A)\leq-\rho_{1}.$ \end{cor}
\begin{proof}
According to \prettyref{thm:characterization_exp_stability}, the
evolutionary problem is exponentially stable with stability rate $\rho_{1}$
if and only if the function 
\[
\mathbb{C}_{\Re>-\rho_{1}}\setminus\{0\}\ni z\mapsto\left(zM(z^{-1})+A\right)^{-1}\in L(H)
\]
is bounded, which is nothing as to say that $\omega_{0}(M,A)\leq-\rho_{1}.$\end{proof}
\begin{cor}
\label{cor:perturbation}Let $M:\mathbb{C}\setminus B[-r,r]\to L(H)$
be analytic and assume that the evolutionary problem 
\[
\left(\partial_{0}M(\partial_{0}^{-1})+A\right)u=f
\]
is well-posed and exponentially stable with stability rate $0<\rho_{1}<\frac{1}{2r}$.
Let 
\[
C\coloneqq\sup_{z\in\mathbb{C}_{\Re>-\rho_{1}}\setminus\{0\}}\left\Vert \left(zM(z^{-1})+A\right)^{-1}\right\Vert ,
\]
which is finite according to \prettyref{thm:characterization_exp_stability}.
Let $N:\mathbb{C}\setminus B[-r,r]\to L(H)$ be analytic and bounded
such that 
\[
\|N\|_{\infty}\coloneqq\sup_{z\in\mathbb{C}\setminus B\left[-\frac{1}{2\rho_{1}},\frac{1}{2\rho_{1}}\right]}\|N(z)\|<\frac{1}{C}.
\]
Then the evolutionary problem 
\[
\left(\partial_{0}\left(M(\partial_{0}^{-1})+\partial_{0}^{-1}N(\partial_{0}^{-1})\right)+A\right)u=f
\]
is well-posed and exponentially stable with stability rate $\rho_{1}.$\end{cor}
\begin{proof}
Using the equality 
\[
zM(z^{-1})+N(z^{-1})+A=\left(N(z^{-1})\left(zM(z^{-1})+A\right)^{-1}+1\right)\left(zM(z^{-1})+A\right)
\]
we obtain that \foreignlanguage{english}{$zM(z^{-1})+N(z^{-1})+A$}
is boundedly invertible with 
\[
\sup_{z\in\mathbb{C}\setminus B\left[-\frac{1}{2\rho_{1}},\frac{1}{2\rho_{1}}\right]}\left\Vert \left(zM(z^{-1})+N(z^{-1})+A\right)^{-1}\right\Vert \leq\frac{C}{1-\|N\|_{\infty}C}.
\]
The assertion now follows from \prettyref{thm:characterization_exp_stability}.
\end{proof}

\section{Second order problems and exponential decay}

Frequently, hyperbolic problems occurring in mathematical physics
are given as a differential equation of second order in time and space.
To obtain an exponential decay one has to assume suitable boundary
conditions, making the spatial operator (e.g. the Dirichlet-Laplacian)
continuously invertible. Following our solution theory, we have to
reformulate the problem as a first order problem. As it turns out,
there are several possibilities to do this, allowing to introduce
an additional parameter. %

We begin to state an exponential stability result for evolutionary
equations, where the spatial operator $A$ is assumed to be invertible.
\begin{prop}
\label{prop:exp_decay_invertible_A}Let $H$ be a Hilbert space, $A:D(A)\subseteq H\to H$
m-accretive and continuously invertible and $r>0$. Moreover, let
$M:\mathbb{C}\setminus B\left[-r,r\right]\to L(H)$ be analytic and
assume that there exists $\delta\in[0,\frac{1}{2r}[$ such that 
\begin{equation}
K\coloneqq\sup_{z\in B[0,\delta]\setminus\{0\}}\|zM(z^{-1})\|<\|A^{-1}\|^{-1}\label{eq:bounded}
\end{equation}
and 
\begin{equation}
\exists c>0,0<\rho_{0}<\frac{1}{2r}\;\forall z\in\mathbb{C}_{\Re>-\rho_{0}}\setminus B[0,\delta]:\Re zM(z^{-1})\geq c.\label{eq:pos_def}
\end{equation}
Then the evolutionary problem 
\[
\left(\partial_{0}M(\partial_{0}^{-1})+A\right)u=f
\]
is well-posed and exponentially stable.\end{prop}
\begin{proof}
By assumption there exist $c>0$ and $0<\rho_{0}<\frac{1}{2r}$ such
that 
\[
\Re zM(z^{-1})\geq c\quad\left(z\in\mathbb{C}_{\Re>-\rho_{0}}\setminus B[0,\delta]\right).
\]
Consequently, $zM(z^{-1})+A$ is continuously invertible for each
$z\in\mathbb{C}_{\Re>-\rho_{0}}\setminus B[0,\delta]$ with 
\[
\left\Vert \left(zM(z^{-1})+A\right)^{-1}\right\Vert \leq\frac{1}{c}.
\]
In particular, this implies that the evolutionary problem is well-posed.
Moreover, for $z\in B[0,\delta]\setminus\{0\}$ we have that 
\[
\|zM(z^{-1})\|\leq K<\|A^{-1}\|^{-1}
\]
and hence, we obtain that $zM(z^{-1})+A$ is continuously invertible
for all $z\in B[0,\delta]\setminus\{0\}$ with 
\[
\left\Vert \left(zM(z^{-1})+A\right)^{-1}\right\Vert \leq\frac{\|A^{-1}\|}{1-K\|A^{-1}\|}.
\]
Thus, we have that 
\[
\mathbb{C}_{\Re>-\rho_{0}}\setminus\{0\}\ni z\mapsto\left(zM(z^{-1})+A\right)^{-1}\in L(H)
\]
is a bounded and analytic mapping and hence, the assertion follows
from \prettyref{thm:characterization_exp_stability}. 
\end{proof}

\subsection{Hyperbolic problems}

We now start from a second order hyperbolic equation of the form 
\begin{equation}
\left(\partial_{0}^{2}M(\partial_{0}^{-1})+C^{\ast}C\right)u=f,\label{eq:hyperbolic}
\end{equation}
where $C:D(C)\subseteq H_{0}\to H_{1}$ is a densely defined closed
linear operator between two Hilbert spaces $H_{0}$ and $H_{1}$,
which is boundedly invertible, $M(\partial_{0}^{-1})=M_{0}(\partial_{0}^{-1})+\partial_{0}^{-1}M_{1}(\partial_{0}^{-1})$,
where $M_{0},M_{1}:\mathbb{C}\setminus B[-r,r]\to L(H_{0})$ are analytic
bounded mappings for some $r>0$. We may rewrite this problem as a
first order system in the new unknowns $v\coloneqq\partial_{0}u+du$
and $q\coloneqq Cu,$ where $d>0$ is an arbitrary parameter. We have
that 
\begin{align*}
\partial_{0}M(\partial_{0}^{-1})v & =\partial_{0}^{2}M(\partial_{0}^{-1})u+d\partial_{0}M(\partial_{0}^{-1})u\\
 & =\partial_{0}^{2}M(\partial_{0}^{-1})u+dM_{0}(\partial_{0}^{-1})\partial_{0}u+dM_{1}(\partial_{0}^{-1})u\\
 & =\partial_{0}^{2}M(\partial_{0}^{-1})u+dM_{0}(\partial_{0}^{-1})\left(v-du\right)+dM_{1}(\partial_{0}^{-1})u
\end{align*}
and, consequently, 
\[
\left(\partial_{0}M(\partial_{0}^{-1})-dM_{0}(\partial_{0}^{-1})\right)v+d\left(dM_{0}(\partial_{0}^{-1})-M_{1}(\partial_{0}^{-1})\right)C^{-1}q+C^{\ast}q=f.
\]
Hence, the resulting system reads as 
\begin{equation}
\left(\partial_{0}\left(\begin{array}{cc}
M(\partial_{0}^{-1}) & 0\\
0 & 1
\end{array}\right)+d\left(\begin{array}{cc}
-M_{0}(\partial_{0}^{-1})\; & \left(dM_{0}(\partial_{0}^{-1})-M_{1}(\partial_{0}^{-1})\right)C^{-1}\\
0 & 1
\end{array}\right)+\left(\begin{array}{cc}
0 & C^{\ast}\\
-C & 0
\end{array}\right)\right)\left(\begin{array}{c}
v\\
q
\end{array}\right)=\left(\begin{array}{c}
f\\
0
\end{array}\right).\label{eq:firstoreder}
\end{equation}

We now consider the new material law $M_{d}(\partial_{0}^{-1})$ depending
on the additional parameter $d>0$ induced by the function
\begin{align}
M_{d}(z) & \coloneqq\left(\begin{array}{cc}
M(z) & 0\\
0 & 1
\end{array}\right)+dz\left(\begin{array}{cc}
-M_{0}(z)\; & \left(dM_{0}(z)-M_{1}(z)\right)C^{-1}\\
0 & 1
\end{array}\right),\label{eq:M_d}
\end{align}
for $z\in\mathbb{C}\setminus B[-r,r]$. The aim is to show that under
suitable assumptions on the material law $M(\partial_{0}^{-1})$,
we can show that $M_{d}$ satisfies the conditions \prettyref{eq:bounded}
and \prettyref{eq:pos_def} of \prettyref{prop:exp_decay_invertible_A}. 
\begin{rem}
\label{rem: first_to_second}Note that if we can show that \prettyref{eq:firstoreder}
is exponentially stable, we get that \prettyref{eq:hyperbolic} is
exponentially stable in the sense that $\partial_{0}u,u\in H_{-\nu}(\mathbb{R};H_{0})$
and $Cu\in H_{-\nu}(\mathbb{R};H_{1})$, if $f\in H_{-\nu}(\mathbb{R};H_{0})$.
Indeed, the exponential stability of \prettyref{eq:firstoreder} yields
$v=\partial_{0}u+du\in H_{-\nu}(\mathbb{R};H_{0})$ and $q=Cu\in H_{-\nu}(\mathbb{R};H_{1}).$
By the bounded invertibility of $C$, we read of that $u\in H_{-\nu}(\mathbb{R};H)$
and hence, $\partial_{0}u=v-du\in H_{-\nu}(\mathbb{R};H).$ Moreover,
we note that we also read off a pointwise decay estimate for $u$
by 
\begin{align*}
|u(t)| & =\left|\intop_{t}^{\infty}\partial_{0}u(s)\mbox{ d}s\right|\\
 & \leq|\partial_{0}u|_{H_{-\nu}(\mathbb{R};H)}e^{-\nu t}\quad(t\in\mathbb{R}),
\end{align*}
which yields $|u(t)|^{\mu t}\to0$ as $t\to\infty$ for each $0<\mu<\nu.$ 
\end{rem}
We begin with the following lemma.
\begin{lem}
\label{lem:pos_def_Md}Let $M_{d}$ be given as in \prettyref{eq:M_d}
and let $z\in\mathbb{C}\setminus B[-r,r]$. If there is $c>0$ such
that $\Re z^{-1}M(z)\geq c$ then 
\[
\Re z^{-1}M_{d}(z)\geq\min\left\{ c-dK(d),\frac{3}{4}d+\Re z\right\} ,
\]
where 
\[
K(d)\coloneqq\|M_{0}\|_{\infty}+\left(d\|M_{0}\|_{\infty}+\|M_{1}\|_{\infty}\|C^{-1}\|\right)^{2}.
\]
\end{lem}
\begin{proof}
By assumption we have $\Re z^{-1}\geq-\frac{1}{2r}$ and we estimate
for $(x,y)\in H_{0}\oplus H_{1}$ 
\begin{align*}
 & \Re\left\langle \left.z^{-1}M_{d}(z)\left(\begin{array}{c}
x\\
y
\end{array}\right)\right|\left(\begin{array}{c}
x\\
y
\end{array}\right)\right\rangle \\
 & \geq c|x|^{2}+d\Re\langle-M_{0}(z)x+\left(dM_{0}(z)-M_{1}(z)\right)C^{-1}y|x\rangle+\left(d+\Re z\right)|y|^{2}\\
 & \geq\left(c-d\|M_{0}\|_{\infty}\right)|x|^{2}-d\left(d\|M_{0}\|_{\infty}+\|M_{1}\|_{\infty}\|C^{-1}\|\right)|x||y|+\left(d+\Re z\right)|y|^{2}\\
 & \geq\left(c-d\|M_{0}\|_{\infty}-\frac{1}{4\varepsilon}d^{2}\left(d\|M_{0}\|_{\infty}+\|M_{1}\|_{\infty}\|C^{-1}\|\right)^{2}\right)|x|^{2}+\left(d+\Re z-\varepsilon\right)|y|^{2},
\end{align*}
for each $\varepsilon>0.$ Choosing $\varepsilon=\frac{d}{4},$ we
obtain the assertion.
\end{proof}
We begin to treat the case, when the function $M$ satisfies \prettyref{eq:pos_def}
for $\delta=0$ (note that condition \prettyref{eq:bounded} is trivially
satisfied for $\delta=0$).
\begin{prop}
\label{prop:damped_wave}Let $M_{d}$ be given as above and assume
that 
\begin{equation}
\exists c>0\forall z\in\mathbb{C}\setminus B[-r,r]:\Re z^{-1}M(z)\geq c.\label{eq:pos_def_M}
\end{equation}
Then there exists $d_{0}>0$ such that the function $M_{d_{0}}$ satisfies
\prettyref{eq:pos_def} for $\delta=0$.\end{prop}
\begin{proof}
Let $z\in\mathbb{C}\setminus B\left[-\frac{1}{2\rho_{0}},\frac{1}{2\rho_{0}}\right]$
where $\rho_{0}\in]0,\frac{1}{2r}[$ will be chosen later. Consequently
$\Re z^{-1}\geq-\rho_{0}$ and we obtain due to \prettyref{lem:pos_def_Md}
\[
\Re z^{-1}M_{d}(z)\geq\min\left\{ c-dK(d),\frac{3}{4}d-\rho_{0}\right\} .
\]
Choosing now $d_{0}$ small enough such that $c-d_{0}K(d_{0})>0$
and then choosing $\rho_{0}<\frac{3}{4}d_{0}$ we derive the positive
definiteness constraint \prettyref{eq:pos_def} for $M_{d_{0}}$ with
$\delta=0$.\end{proof}
\begin{example}
The latter proposition applies to material laws of the form $M(z)=M_{0}+zM_{1},$
where $M_{0}\in L(H)$ is a selfadjoint, non-negative operator, $M_{1}\in L(H)$
is strictly positive definite and $z\in\mathbb{C}$. Indeed, $M_{0}(z)=M_{0}$
and $M_{1}(z)=M_{1}$ are trivially analytic and bounded mappings
on $\mathbb{C}$ and we have that 
\[
\Re z^{-1}M(z)=\Re z^{-1}M_{0}+\Re M_{1}\geq\Re z^{-1}M_{0}+c
\]
for each $z\in\mathbb{C}\setminus\{0\}$, where $c>0$ is such that
$\Re M_{1}\geq c$. Hence, choosing $r<\frac{\|M_{0}\|}{2c}$ we obtain
that \prettyref{eq:pos_def_M} holds. The corresponding second order-problem
is 
\begin{equation}
\left(\partial_{0}^{2}M_{0}+\partial_{0}M_{1}+C^{\ast}C\right)u=f,\label{eq:wave_heat}
\end{equation}
which therefore is exponentially stable due to \prettyref{prop:exp_decay_invertible_A}.
Note that the above equation is not simply an abstract damped wave
equation, since $M_{0}$ is allowed to have a non-trivial kernel.
Indeed, let $H=L_{2}(\Omega)$ for some bounded domain $\Omega\subseteq\mathbb{R}^{n}$,
$-C^{\ast}C=\Delta_{D}$ the Dirichlet-Laplacian, $M_{0}=\chi_{\Omega_{1}}(\mathrm{m})$
for some $\Omega_{1}\subseteq\Omega$ and $M_{1}=c>0$. Then \prettyref{eq:wave_heat}
is a combination of the heat equation on $\Omega\setminus\Omega_{1}$
(as $M_{0}=0$ on this set) and the damped wave equation on $\Omega_{1}$
(as $M_{0}=1$ on $\Omega_{1}$). Such classes of degenerated differentail
equations have been studied for instance in \cite{Colli1995,Colli1996},
see also the monograph \cite{Favini1999}.
\end{example}
Of course, we want to apply \prettyref{prop:exp_decay_invertible_A}
to a broader class of hyperbolic differential equations such as integro-differential
equations or delay equations. It turns out, that in this case the
material law $M(\partial_{0}^{-1})$ fails to satisfy the condition
\prettyref{eq:pos_def_M}. One way to deal with such equations provides
our next result.
\begin{prop}
\label{prop:exp_decay_integro}Let $M_{d}$ be given as in \prettyref{eq:M_d}
and assume that 
\begin{equation}
\forall\delta>0\;\exists\rho_{0}\in\left]0,\frac{1}{2r}\right[,c>0\;\forall z\in\mathbb{C}_{\Re>-\rho_{0}}\setminus B[0,\delta]:\:\Re zM(z^{-1})\geq c.\label{eq:pos_def_M_locally}
\end{equation}
Moreover, we assume that $\lim_{z\to0}M_{1}(z^{-1})=0$. Then there
exists $d_{0}>0$ such that the material law $M_{d_{0}}$ satisfies
\prettyref{eq:bounded} and \prettyref{eq:pos_def} for a suitable
$\delta>0$, where $A\coloneqq\left(\begin{array}{cc}
0 & C^{\ast}\\
-C & 0
\end{array}\right).$ \end{prop}
\begin{proof}
As $M_{0}$ and $M_{1}$ are assumed to be bounded, we obtain that
\[
G(d)\coloneqq\sup_{z\in\mathbb{C}\setminus B[-r,r]}\left\Vert d\left(\begin{array}{cc}
-M_{0}(z)\; & \left(dM_{0}(z)-M_{1}(z)\right)C^{-1}\\
0 & 1
\end{array}\right)\right\Vert \to0\quad(d\to0).
\]
Hence, recalling that $M(z^{-1})=M_{0}(z^{-1})+zM_{1}(z^{-1})$, we
estimate for $z\in\mathbb{C}_{\Re>-\frac{1}{2r}}\setminus\{0\}$ 
\begin{align*}
\left\Vert zM_{d}(z^{-1})\right\Vert  & \leq\max\{\|zM(z^{-1})\|,|z|\}+G(d)\\
 & =\max\left\{ |z|\|M_{0}\|_{\infty}+\|M_{1}(z^{-1})\|,|z|\right\} +G(d).
\end{align*}
Thus, choosing $\delta>0$ and $d_{1}>0$ small enough we obtain that
\[
\sup_{z\in B[0,\delta]\setminus\{0\}}\|zM_{d}(z^{-1})\|\leq\max\left\{ \delta\|M_{0}\|_{\infty}+\sup_{z\in B[0,\delta]\setminus\{0\}}\|M_{1}(z^{-1})\|,\delta\right\} +G(d_{1})<\|A^{-1}\|^{-1}
\]
for each $0<d<d_{1}.$ This shows \prettyref{eq:bounded} for $M_{d}$
where $0<d<d_{1}$. According to \prettyref{eq:pos_def_M_locally}
there is $c>0$ and $\rho_{0}\in]0,\frac{1}{2r}[$ such that 
\[
\Re zM(z^{-1})\geq c
\]
for all $z\in\mathbb{C}_{\Re>-\rho_{0}}\setminus B[0,\delta].$ Thus,
by \prettyref{lem:pos_def_Md} we have that 
\[
\Re zM_{d}(z^{-1})\geq\min\left\{ c-dK(d),\frac{3}{4}d-\rho_{1}\right\} 
\]
for every $z\in\mathbb{C}_{\Re>-\rho_{1}}\setminus B[0,\delta]$,
where $0<\rho_{1}\leq\rho_{0}$. Hence, choosing first $0<d_{0}\leq d_{1}$
small enough such that $d_{0}K(d_{0})<c$ and then $\rho_{1}<\frac{3}{4}d_{0}$
we derive \prettyref{eq:pos_def} for $M_{d_{0}}.$
\end{proof}

\subsection{Integro-differential equations\label{sub:integro}}

A class of material laws $M(\partial_{0}^{-1})$ satisfying the assumptions
of \prettyref{prop:exp_decay_integro} arises in the study of hyperbolic
integro-differential equations of the form 
\[
\left(\partial_{0}^{2}\left(1-k\ast\right)^{-1}+C^{\ast}C\right)u=f.
\]
Following \cite{Trostorff2012_integro}, we consider kernels $k:\mathbb{R}_{\geq0}\to L(H_{0})$
with the following properties

\begin{hyp} $\,$

\begin{enumerate}[(a)]

\item $k$ is weakly measurable, i.e. for each $x,y\in H_{0}$, the
function $\mathbb{R}_{\geq0}\ni t\mapsto\langle k(t)x|y\rangle$ is
measurable,

\item $\mathbb{R}_{\geq0}\ni t\mapsto\|k(t)\|$ is measurable,%
\footnote{Note that this follows from (a) if $H_{0}$ is separable.%
}

\item there exists $\alpha>0$ such that 
\[
|k|_{L_{1,-\alpha}}\coloneqq\intop_{0}^{\infty}e^{\alpha t}\|k(t)\|\mbox{ d}t<1,
\]

\item $k(t)$ is selfadjoint for almost every $t\in\mathbb{R}_{\geq0},$

\item $k(t)k(s)=k(s)k(t)$ for almost every $t,s\in\mathbb{R}_{\geq0}.$

\end{enumerate}

\end{hyp}

We consider the material law $M(z)=(1-\sqrt{2\pi}\:\hat{k}(-\i z^{-1}))^{-1}$
for $z\in\mathbb{C}\setminus B[-r,r],$ where $r=\frac{1}{2\alpha}$
and 
\[
\hat{k}(z)=\frac{1}{\sqrt{2\pi}}\intop_{0}^{\infty}e^{-\i zt}k(t)\mbox{ d}t,\quad(z\in\mathbb{C}_{\Im<\alpha})
\]
where the integral is meant in the weak sense. The operator $M(\partial_{0}^{-1})$
is then $(1-k\ast)^{-1}$ (see \cite{Trostorff2012_integro} for more
details). This material law is of the form $M(z)=M_{0}(z)+zM_{1}(z)$
with $M_{1}(z)=0$, since $\|M(z)\|\leq\frac{1}{1-|k|_{L_{1,-\alpha}}}$
for all $z\in\mathbb{C}\setminus B[-r,r].$ In order to obtain \prettyref{eq:pos_def_M_locally}
for this material law $M$ one needs to assume a similar estimate
for the imaginary part of the Fourier transform $\hat{k}$, which
reads as follows: 

\begin{hyp}$\,$

(f) For each $\delta>0$ there is a function $g:\mathbb{R}_{>-\alpha}\to\mathbb{R}_{\geq0}$
continuous at $0$ with $g(0)>0$ such that for each $|t|>\delta,\rho>-\alpha$
we have the estimate%
\footnote{Recall, that for a bounded operator $T\in L(H)$ its imaginary part
$\Im T$ is defined as the selfadjoint operator $\frac{1}{2\i}(T-T^{\ast})$.
Thus, Hypothesis (f) means, using assumption (a), that 
\[
\left\langle \left.\frac{1}{2\i}t\left(\hat{k}(t-\i\rho)-\hat{k}(-t-\i\rho)\right)x\right|x\right\rangle \leq-g(\rho)|x|^{2}.\quad(x\in H_{0},|t|>\delta,\rho>-\alpha)
\]
} 
\[
t\Im\hat{k}(t-\i\rho)\leq-g(\rho).
\]

\end{hyp}

We can now prove that under the Hypotheses (a)-(f), that the material
law $M(z)=(1-\sqrt{2\pi}\:\hat{k}(-\i z^{-1}))^{-1}$ satisfies \prettyref{eq:pos_def_M_locally}.
For doing so, let $\delta>0$ and $z\in\mathbb{C}_{\Re>-\rho_{0}}\setminus[-\delta,\delta]^{2},$
$z=\i t+\rho$ for $|t|>\delta$ and $\rho\geq-\rho_{0}$, where $\rho_{0}\in]0,\alpha[$
will be chosen later. We set 
\[
D\coloneqq\left|1-\sqrt{2\pi}\:\hat{k}(t-\i\rho)\right|^{-1}
\]
and note that due to assumption (e) the operator $\hat{k}(t-\i\rho)$
is normal and thus, $D$ and $1-\sqrt{2\pi}\:\hat{k}(-t-\i\rho)=\left(1-\sqrt{2\pi}\:\hat{k}(t-\i\rho\right)^{\ast}$
commute. We estimate for $x\in H_{0}:$ 
\begin{align}
\Re\langle zM(z^{-1})x|x\rangle & =\Re(\i t+\rho)\langle\left(1-\sqrt{2\pi}\:\hat{k}(t-\i\rho)\right)^{-1}x|x\rangle\nonumber \\
 & \geq\rho\Re\langle\left(1-\sqrt{2\pi}\:\hat{k}(t-\i\rho)\right)^{-1}x|x\rangle-t\Im\langle\left(1-\sqrt{2\pi}\:\hat{k}(t-\i\rho)\right)^{-1}x|x\rangle.\nonumber \\
 & =\rho\Re\langle\left(1-\sqrt{2\pi}\:\hat{k}(-t-\i\rho)\right)Dx|Dx\rangle-t\Im\langle\left(1-\sqrt{2\pi}\:\hat{k}(-t-\i\rho)\right)Dx|Dx\rangle\nonumber \\
 & =\rho\Re\langle\left(1-\sqrt{2\pi}\:\hat{k}(-t-\i\rho)\right)Dx|Dx\rangle-\sqrt{2\pi}\langle t\Im\hat{k}(t-\i\rho)Dx|Dx\rangle\nonumber \\
 & \geq\rho\Re\langle\left(1-\sqrt{2\pi}\:\hat{k}(-t-\i\rho)\right)Dx|Dx\rangle+\sqrt{2\pi}g(\rho)|Dx|^{2}.\label{eq:estimate integro}
\end{align}
If $\rho\in[-\rho_{0},\rho_{0}]$ for some $0<\rho_{0}<\alpha$, the
latter term can be estimated by 
\[
\frac{-\rho_{0}\left(1+|k|_{1,-\alpha}\right)+\sqrt{2\pi}\inf_{\rho\in[-\rho_{0},\rho_{0}]}g(\rho)}{(1+|k|_{1,-\alpha})^{2}}|x|^{2},
\]
where we have used 
\[
|x|=|D^{-1}Dx|\leq(1+|k|_{1,-\alpha})|Dx|.
\]
Since $\frac{-\rho_{0}\left(1+|k|_{1,-\alpha}\right)+\sqrt{2\pi}\inf_{\rho\in[-\rho_{0},\rho_{0}]}g(\rho)}{(1+|k|_{1,-\alpha})^{2}}\to\frac{\sqrt{2\pi}g(0)}{(1+|k|_{1,-\alpha})^{2}}>0$
as $\rho_{0}\to0,$ we find $\rho_{0}\in]0,\alpha[$ such that 
\[
\frac{-\rho_{0}\left(1+|k|_{1,-\alpha}\right)+\sqrt{2\pi}\inf_{\rho\in[-\rho_{0},\rho_{0}]}g(\rho)}{(1+|k|_{1,-\alpha})^{2}}>0.
\]
If $\rho>\rho_{0}$ we have that \prettyref{eq:estimate integro}
can be estimated by 
\[
\frac{\rho\left(1-|k|_{1,-\alpha}\right)+\sqrt{2\pi}g(\rho)}{(1+|k|_{1,-\alpha})^{2}}|x|^{2}\geq\frac{\rho_{0}\left(1-|k|_{1,-\alpha}\right)}{(1+|k|_{1,-\alpha})^{2}}|x|^{2}.
\]
Summarizing, we have shown that \prettyref{eq:pos_def_M_locally}
holds for our material law $M$ and hence, the corresponding evolutionary
equation 
\[
\left(\partial_{0}^{2}\left(1-k\ast\right)^{-1}+C^{\ast}C\right)u=f
\]
is exponentially stable by \prettyref{prop:exp_decay_integro}.
\begin{rem}
\label{rem:Cannarsa}In \cite{Cannarsa2011} the authors consider
scalar-valued kernels $k\in L_{1,-\alpha}(\mathbb{R}_{\geq0};\mathbb{R})$
for some $\alpha>0$, such that $t\mapsto\intop_{t}^{\infty}e^{\alpha s}k(s)\mbox{ d}s$
defines a strongly positive definite kernel on $L_{2,\loc}(\mathbb{R}_{\geq0})$
and $\intop_{0}^{\infty}e^{\alpha s}|k(s)|\mbox{ d}s<1$. Clearly,
these kernels satisfy the hypotheses (a)-(e). Moreover, following
\cite[Proposition 2.2, Proposition 2.5]{Cannarsa2011} these kernels
satisfy an estimate of the form 
\[
\exists c>0\:\forall t>0,\rho>-\alpha:\intop_{0}^{\infty}\sin(ts)e^{-\rho s}k(s)\mbox{ d}t\geq c\frac{1}{(\alpha+\rho+1)^{2}}\frac{t}{1+t^{2}}.
\]
Hence, 
\begin{align*}
t\Im\hat{k}(t-\i\rho) & =t\frac{1}{\sqrt{2\pi}}\intop_{0}^{\infty}\sin(-ts)e^{-\rho s}k(s)\mbox{ d}s\\
 & =-t\frac{1}{\sqrt{2\pi}}\intop_{0}^{\infty}\sin(ts)e^{-\rho s}k(s)\mbox{ d}t\\
 & \leq-\frac{c}{\sqrt{2\pi}(\alpha+\rho+1)^{2}}\frac{t^{2}}{1+t^{2}}
\end{align*}
for each $t\geq0,\rho>-\alpha.$ Since $\Im\hat{k}(-t-\i\rho)=-\Im\hat{k}(t-\i\rho),$
we obtain 
\[
t\Im\hat{k}(t-\i\rho)\leq-\frac{c}{\sqrt{2\pi}(\alpha+\rho+1)^{2}}\frac{t^{2}}{1+t^{2}}
\]
for every $t\in\mathbb{R},\rho>-\alpha.$ Let $\delta>0$ and $|t|>\delta.$
Then 
\[
t\Im\hat{k}(t-\i\rho)\leq-\frac{c}{\sqrt{2\pi}(\alpha+\rho+1)^{2}}\frac{\delta^{2}}{1+\delta^{2}}
\]
for $\rho>-\alpha.$ As $g=\left(\rho\mapsto\frac{c}{\sqrt{2\pi}(\alpha+\rho+1)^{2}}\frac{\delta^{2}}{1+\delta^{2}}\right)$
is continuous and attains positive values only, we have that $k$
satisfies (f) and thus, the exponential stability of the corresponding
evolutionary equation is covered by our theory.
\end{rem}

\section{The wave equation with a time delay and a convolution integral}

Motivated by a recent paper of Alabau-Boussouira et al. \cite{Alabau2014},
we study the following wave equation 
\begin{equation}
\partial_{0}^{2}u-(1-k\ast)\Delta u+\kappa\tau_{-h}\partial_{0}u=f\label{eq:wave}
\end{equation}
on a domain $\Omega\subseteq\mathbb{R}^{n}$ with homogeneous Dirichlet
boundary conditions, i.e. 
\begin{equation}
u=0\mbox{ on }\partial\Omega.\label{eq:bd_cond}
\end{equation}
Here $\tau_{-h}$ denotes the translation operator for some $h>0$,
i.e. $\left(\tau_{-h}f\right)(t)=f(t-h)$ and $\kappa\in\mathbb{R}$
is a given parameter. In \cite{Alabau2014}, the kernel $k:\mathbb{R}_{\geq0}\to\mathbb{R}_{\geq0}$
is assumed to be locally absolutely continuous, $k(0)>0,$ $\intop_{0}^{\infty}k(s)\mbox{ d}s<1$
and $k'(t)\leq-\alpha k(t)$ for some $\alpha>0$ and each $t>0$.
We will generalize this to operator-valued kernels satisfying the
assumptions given in \prettyref{sub:integro}. Moreover, we will show
that \prettyref{eq:wave} fits into our general setting and that the
exponential stability can be shown under the assumption that $|\kappa|$
is sufficiently small. This is exactly the result stated in \cite{Alabau2014},
however under weaker assumptions on the kernel $k$ and by a completely
different approach. First of all, we show that \prettyref{eq:wave}
is indeed of the form \prettyref{eq:hyperbolic}. For doing so, we
need to introduce the spatial differential operators involved.
\begin{defn*}
Let $\Omega\subseteq\mathbb{R}^{n}$. We define the gradient with
vanishing boundary values $\grad_{c}$ as the closure of 
\begin{align*}
\grad_{c}|_{C_{c}^{\infty}(\Omega)}:C_{c}^{\infty}(\Omega)\subseteq L_{2}(\Omega) & \to L_{2}(\Omega)^{n}\\
\phi & \mapsto\left(\partial_{i}\phi\right)_{i\in\{1,\ldots,n\}}.
\end{align*}
Moreover, we define $\dive\coloneqq-\grad_{c}^{\ast}$, the divergence
with maximal domain in $L_{2}(\Omega)^{n}.$\end{defn*}
\begin{rem}
We note that by definition, the domain of $\grad_{c}$ is nothing
but the well-known Sobolev-space $H_{0}^{1}(\Omega)$ -- the closure
of the test function $C_{c}^{\infty}(\Omega)$ with respect to the
topology on $H^{1}(\Omega)$. In consequence, the domain of $\dive$
is 
\[
D(\dive)=\left\{ \Phi=(\Phi_{i})_{i\in\{1,\ldots,n\}}\in L_{2}(\Omega)^{n}\,|\,\dive\Phi=\sum_{i=1}^{n}\partial_{i}\Phi_{i}\in L_{2}(\Omega)\right\} ,
\]
where $\partial_{i}\Phi_{i}$ is meant in the distributional sense. 
\end{rem}
Using these operators, \prettyref{eq:wave} together with the boundary
condition \prettyref{eq:bd_cond} reads as 
\begin{equation}
\partial_{0}^{2}u-(1-k\ast)\dive\grad_{c}u+\kappa\tau_{-h}\partial_{0}u=f.\label{eq:wave_new}
\end{equation}

From now on we assume that the kernel $k$ is operator-valued, i.e.
$k:\mathbb{R}_{\geq0}\to L(L_{2}(\Omega))$, and satisfies the hypotheses
(a)-(f) of \prettyref{sub:integro}. The next lemma shows that this
is indeed a generalization of the assumptions on $k$ made in \cite{Alabau2014}.
\begin{lem}
Let $k:\mathbb{R}_{\geq0}\to\mathbb{R}_{\geq0}$ be locally absolutely
continuous, $\intop_{0}^{\infty}k(s)\,\mathrm{d}s<1,$ $k(0)>0$ and
$k'(t)\leq-\alpha_{0}k(t)$ for some $\alpha_{0}>0$ and each $t\in\mathbb{R}_{\geq0}$.
Then $k$ satisfies (a)-(f) in \prettyref{sub:integro}. \end{lem}
\begin{proof}
The Hypotheses (a),(b),(d) and (e) are trivially satisfied. Moreover,
we observe that $k'(t)\leq-\alpha_{0}k(t)$ yields that $k(t)\leq k(0)e^{-\alpha_{0}t}$
for each $t\in\mathbb{R}_{\geq0}$ and thus, 
\[
\intop_{0}^{\infty}k(t)e^{\alpha t}\mbox{ d}t<\infty\quad(\alpha<\alpha_{0}).
\]
Since 
\[
\mathbb{R}_{<\alpha_{0}}\ni\alpha\mapsto\intop_{0}^{\infty}k(t)e^{\alpha t}\mbox{ d}t
\]
is continuous, it follows form $\intop_{0}^{\infty}k(s)\mbox{ d}s<1$
that (c) holds. Moreover, we have that 
\[
\intop_{0}^{N}e^{\alpha t}|k'(t)|\mbox{ d}t=-\intop_{0}^{N}e^{\alpha t}k'(t)\mbox{ d}t=-e^{\alpha N}k(N)+k(0)+\alpha\intop_{0}^{N}e^{\alpha t}k(t)\mbox{ d}t\quad(\alpha<\alpha_{0},\, N\in\mathbb{N}),
\]
which gives $k'\in L_{1,-\alpha}(\mathbb{R}_{\geq0})$, as the latter
term converges to $k(0)+\alpha|k|_{1,-\alpha}$ as $N$ tends to infinity.
It is left to show that $k$ satisfies Hypothesis (f) of \prettyref{sub:integro}.
Let $0<\alpha<\alpha_{0}$. We claim that 
\begin{equation}
\exists c>0\,\forall t>0:\Im\hat{k}(t+\i\alpha)\leq-c\frac{t}{1+t^{2}}.\label{eq:kernel_est}
\end{equation}
Let us consider the function 
\[
\Phi(t)\coloneqq-\frac{1+t^{2}}{t}\Im\hat{k}(t+\i\alpha)=\frac{1}{\sqrt{2\pi}}\frac{1+t^{2}}{t}\intop_{0}^{\infty}\sin(ts)e^{\alpha s}k(s)\mbox{ d}s\quad(t>0).
\]
Obviously, $\Phi$ is continuous. Moreover, $\Phi(t)>0$ for each
$t>0.$ Indeed, we compute 
\begin{align*}
\intop_{0}^{\infty}\sin(ts)e^{\alpha s}k(s)\mbox{ d}s & =\frac{1}{t}\left(\intop_{0}^{\infty}\cos(ts)\left(\alpha k(s)+k'(s)\right)e^{\alpha s}\mbox{ d}s+k(0)\right)\\
 & =\frac{1}{t}\intop_{0}^{\infty}\left(1-\cos(ts)\right)\left(-\alpha k(s)-k'(s)\right)e^{\alpha s}\mbox{ d}s\\
 & \geq\frac{1}{t}\left(\alpha_{0}-\alpha\right)\intop_{0}^{\infty}\left(1-\cos(ts)\right)k(s)e^{\alpha s}\mbox{ d}s
\end{align*}
and since the integrand in the latter integral is positive (except
for $ts=(2j+1)\frac{\pi}{2},\: j\in\mathbb{N})$, we derive that $\Phi(t)>0$.
Moreover, using the latter computation we have that 
\begin{align*}
\Phi(t) & \geq\left(\alpha_{0}-\alpha\right)\frac{1+t^{2}}{t^{2}}\frac{1}{\sqrt{2\pi}}\intop_{0}^{\infty}\left(1-\cos(ts)\right)k(s)e^{\alpha s}\mbox{ d}s\\
 & =\left(\alpha_{0}-\alpha\right)\frac{1+t^{2}}{t^{2}}\left(\frac{1}{\sqrt{2\pi}}|k|_{1,\alpha}-\Re\hat{k}(t+\i\alpha)\right).
\end{align*}
Thus, by the Riemann-Lebesgue Lemma we get $\liminf_{t\to\infty}\Phi(t)\geq\frac{\left(\alpha_{0}-\alpha\right)}{\sqrt{2\pi}}|k|_{1,\alpha}>0$.
Furthermore, using the rule of l'Hospital, we compute 
\begin{align*}
\lim_{t\to0}\Phi(t) & =\lim_{t\to0}\frac{1}{t}\frac{1}{\sqrt{2\pi}}\intop_{0}^{\infty}\sin(ts)e^{\alpha s}k(s)\mbox{ d}s\\
 & =\lim_{t\to0}\frac{1}{\sqrt{2\pi}}\intop_{0}^{\infty}\cos(ts)se^{\alpha s}k(s)\mbox{ d}s\\
 & =\frac{1}{\sqrt{2\pi}}\intop_{0}^{\infty}se^{\alpha s}k(s)\mbox{ d}s>0,
\end{align*}
where we have used that $\left(s\mapsto se^{\alpha s}k(s)\right)\in L_{1}(\mathbb{R}_{\geq0}),$
which follows since $k(s)\leq k(0)e^{-\alpha_{0}s}$ for $s\geq0.$
Summarizing, we have shown that $\Phi:\mathbb{R}_{>0}\to\mathbb{R}_{>0}$
is a continuous function with $\lim_{t\to0}\Phi(t)>0$ and $\lim\inf_{t\to\infty}\Phi(t)>0$
and hence, there is a constant $c>0$ with $\Phi(t)\geq c$ for each
$t>0.$ This proves \prettyref{eq:kernel_est}. Finally, by \cite[Proposition 2.5]{Cannarsa2011},
estimate \prettyref{eq:kernel_est} implies 
\[
\forall t>0,\rho>-\alpha:\Im\hat{k}(t-\i\rho)\leq-\frac{c}{(1+\alpha+\rho)^{2}}\frac{t}{1+t^{2}},
\]
which yields (f) (cp. \prettyref{rem:Cannarsa}).
\end{proof}
We come back to the exponential stability of \prettyref{eq:wave_new}.
Since we have assumed that $k$ satisfies Hypothesis (c), we get that
$(1-k\ast)$ is a boundedly invertible operator on $H_{\rho}(\mathbb{R};L_{2}(\Omega))$
for each $\rho\geq-\alpha.$ Consequently, \prettyref{eq:wave_new}
may be written as%
\footnote{Note that $\tilde{f}\in H_{\rho}(\mathbb{R};L_{2}(\Omega))$ if $f\in H_{\rho}(\mathbb{R};L_{2}(\Omega))$.%
} 
\begin{equation}
\left(\partial_{0}^{2}\left(1-k\ast\right)^{-1}-\dive\grad_{c}+\kappa(1-k\ast)^{-1}\tau_{-h}\partial_{0}\right)u=(1-k\ast)^{-1}f\eqqcolon\tilde{f}.\label{eq:wave_2}
\end{equation}
We assume that $\grad_{c}$ is injective and has closed range, which
can for instance be achieved by assuming that $\Omega$ is bounded
due to Poincar�'s inequality. We denote by $\iota$ the canonical
embedding of $R(\grad_{c})$ into $L_{2}(\Omega)^{n}$, i.e. 
\begin{align*}
\iota:R(\grad_{c}) & \to L_{2}(\Omega)^{n}\\
F & \mapsto F.
\end{align*}
Consequently, $\iota^{\ast}:L_{2}(\Omega)^{n}\to R(\grad_{c})$ is
the orthogonal projection onto the space $R(\grad_{c})$ and we get
$-\dive\grad_{c}=-\dive\iota\iota^{\ast}\grad_{c}.$ Setting $C:D(\grad_{c})\subseteq L_{2}(\Omega)\to R(\grad_{c})$
defined by $C=-\iota^{\ast}\grad_{c}$ we get that $C$ is boundedly
invertible by the closed graph theorem and $C^{\ast}=\dive\iota$
(see e.g. \cite[Lemma 2.4]{Trostorff2012_eeliptic}). Hence, \prettyref{eq:wave_2}
is of the form \prettyref{eq:hyperbolic} with 
\begin{align*}
M_{0}(z) & =\left(1-\sqrt{2\pi}k(-\i z^{-1})\right)^{-1},\\
M_{1}(z) & =\kappa\left(1-\sqrt{2\pi}\hat{k}(-\i z^{-1})\right)^{-1}e^{-hz^{-1}},\\
C & =\iota^{\ast}\grad_{c}.
\end{align*}
Unfortunately, the material law $M(\partial_{0}^{-1})=M_{0}(\partial_{0}^{-1})+\partial_{0}^{-1}M_{1}(\partial_{0}^{-1})$
does not satisfy the assumption of \prettyref{prop:damped_wave} or
\prettyref{prop:exp_decay_integro}. But we can prove the exponential
stability by a perturbation argument. Let us first recall, how to
write \prettyref{eq:wave_2} as a first order system. We define the
material law $M_{d}$ for some $d>0$ by 
\begin{align}
M_{d}(z) & =\left(\begin{array}{cc}
M(z) & 0\\
0 & 1
\end{array}\right)+dz\left(\begin{array}{cc}
-M_{0}(z)\; & \left(dM_{0}(z)-M_{1}(z)\right)C^{-1}\\
0 & 1
\end{array}\right)\label{eq:M_d-1}\\
 & =\left(\begin{array}{cc}
M_{0}(z) & 0\\
0 & 1
\end{array}\right)+dz\left(\begin{array}{cc}
-M_{0}(z)\; & \left(dM_{0}(z)\right)C^{-1}\\
0 & 1
\end{array}\right)+z\left(\left(\begin{array}{cc}
M_{1}(z) & 0\\
0 & 0
\end{array}\right)-d\left(\begin{array}{cc}
0 & M_{1}(z)C^{-1}\\
0 & 0
\end{array}\right)\right)\nonumber 
\end{align}
and set 
\begin{align*}
\tilde{M}_{d}(z) & \coloneqq\left(\begin{array}{cc}
M_{0}(z) & 0\\
0 & 1
\end{array}\right)+dz\left(\begin{array}{cc}
-M_{0}(z)\; & \left(dM_{0}(z)\right)C^{-1}\\
0 & 1
\end{array}\right),\\
N_{d}(z) & \coloneqq\left(\begin{array}{cc}
M_{1}(z) & 0\\
0 & 0
\end{array}\right)-d\left(\begin{array}{cc}
0 & M_{1}(z)C^{-1}\\
0 & 0
\end{array}\right).
\end{align*}
Thus, the first order formulation of \prettyref{eq:wave_2} can be
written as 
\begin{equation}
\left(\partial_{0}\tilde{M}_{d}(\partial_{0}^{-1})+N_{d}(\partial_{0}^{-1})+\left(\begin{array}{cc}
0 & C^{\ast}\\
-C & 0
\end{array}\right)\right)\left(\begin{array}{c}
v\\
q
\end{array}\right)=\left(\begin{array}{c}
\tilde{f}\\
0
\end{array}\right),\label{eq:wave_first_order}
\end{equation}
where $v\coloneqq\partial_{0}u+du$ and $q\coloneqq Cu.$ We note
that according to our findings in \prettyref{sub:integro} and by
\prettyref{prop:exp_decay_invertible_A}, there is $d>0$ such that
\[
\left(\partial_{0}\tilde{M}_{d}(\partial_{0}^{-1})+\left(\begin{array}{cc}
0 & C^{\ast}\\
-C & 0
\end{array}\right)\right)\left(\begin{array}{c}
v\\
q
\end{array}\right)=\left(\begin{array}{c}
\tilde{f}\\
0
\end{array}\right)
\]
is exponentially stable, since the kernel $k$ is assumed to satisfy
the Hypotheses (a)-(f). We denote its stability rate by $\rho_{1}$.
Moreover, since 
\[
\|N_{d}(z)\|\leq\sqrt{1+d^{2}\|C^{-1}\|^{2}}\|M_{1}(z)\|\leq\sqrt{1+d^{2}\|C^{-1}\|^{2}}|\kappa|\left(1-|k|_{1,-\rho_{1}}\right)^{-1}e^{h\rho_{1}},
\]
for each $z\in\mathbb{C}\setminus B\left[-\frac{1}{2\rho_{1}},\frac{1}{2\rho_{1}}\right],$
we may choose $|\kappa|$ small enough to get 
\[
\sup_{z\in\mathbb{C}\setminus B\left[-\frac{1}{2\rho_{1}},\frac{1}{2\rho_{1}}\right]}\|N_{d}(z)\|<\sup_{z\in\mathbb{C}\setminus B\left[-\frac{1}{2\rho_{1}},\frac{1}{2\rho_{1}}\right]}\left\Vert \left(z^{-1}\tilde{M}_{d}(z)+\left(\begin{array}{cc}
0 & C^{\ast}\\
-C & 0
\end{array}\right)\right)^{-1}\right\Vert .
\]
Then, by \prettyref{cor:perturbation} we derive the exponential stability
of \prettyref{eq:wave_first_order}. We summarize our findings in
the following theorem.
\begin{thm}
Let $\Omega\subseteq\mathbb{R}^{n}$ be a domain, such that the Poincar�
inequality holds, i.e. 
\[
\exists c>0\:\forall u\in D(\grad_{c}):|u|_{L_{2}(\Omega)}\leq c|\grad_{c}u|_{L_{2}(\Omega)^{n}}.
\]
Moreover, let $k:\mathbb{R}_{\geq0}\to L(L_{2}(\Omega)^{n})$ be a
kernel that satisfies the Hypotheses (a)-(f) of \prettyref{sub:integro}
and let $h>0$. Then there exists a $\kappa_{0}>0$ such that for
each $|\kappa|<\kappa_{0}$ the problem 
\[
\partial_{0}^{2}u-(1-k\ast)\dive\grad_{c}u+\kappa\tau_{-h}\partial u=f
\]
is exponentially stable (cp. \prettyref{rem: first_to_second}).
\end{thm}

\section{Acknowledgement}

We thank Marcus Waurick for carefully reading the text and for fruitful
discussions, especially about the proof of \prettyref{thm:characterization_exp_stability}.


\begin{thebibliography}{10}

\bibitem{Alabau2014}
F.~Alabau-Boussouira, S.~Nicaise, and C.~Pignotti.
\newblock {Exponential stability of the wave equation with memory and time
  delay.}
\newblock Technical report, 2014.
\newblock arXiv:1404.4456.

\bibitem{Cannarsa2011}
P.~Cannarsa and D.~Sforza.
\newblock {Integro-differential equations of hyperbolic type with positive
  definite kernels.}
\newblock {\em J. Differ. Equations}, 250(12):4289--4335, 2011.

\bibitem{Colli1995}
P.~{Colli} and A.~{Favini}.
\newblock {On some degenerate second order equations of mixed type.}
\newblock {\em {Funkc. Ekvacioj, Ser. Int.}}, 38(3):473--489, 1995.

\bibitem{Colli1996}
P.~{Colli} and A.~{Favini}.
\newblock {Time discretization of nonlinear Cauchy problems applying to mixed
  hyperbolic-parabolic equations.}
\newblock {\em {Int. J. Math. Math. Sci.}}, 19(3):481--494, 1996.

\bibitem{daPrato1975}
G.~{da Prato} and P.~{Grisvard}.
\newblock {Sommes d'op\'erateurs lin\'eaires et \'equations diff\'erentielles
  op\'erationnelles.}
\newblock {\em {J. Math. Pures Appl. (9)}}, 54:305--387, 1975.

\bibitem{Datko_1970}
R.~Datko.
\newblock {Extending a theorem of A. M. Liapunov to Hilbert space.}
\newblock {\em J. Math. Anal. Appl.}, 32:610--616, 1970.

\bibitem{engel2000one}
K.~J. Engel and R.~Nagel.
\newblock {\em {One-parameter semigroups for linear evolution equations}}.
\newblock Graduate texts in mathematics. Springer, 2000.

\bibitem{Favini1999}
A.~{Favini} and A.~{Yagi}.
\newblock {\em {Degenerate differential equations in Banach spaces.}}
\newblock New York, NY: Marcel Dekker, 1999.

\bibitem{Gearhart_1978}
L.~Gearhart.
\newblock {Spectral theory for contraction semigroups on Hilbert space.}
\newblock {\em Trans. Am. Math. Soc.}, 236:385--394, 1978.

\bibitem{Kalauch2011}
A.~{Kalauch}, R.~{Picard}, S.~{Siegmund}, S.~{Trostorff}, and M.~{Waurick}.
\newblock {A Hilbert space perspective on ordinary differential equations with
  memory term.}
\newblock {\em {J. Dyn. Differ. Equations}}, 26(2):369--399, 2014.

\bibitem{kato1995perturbation}
T.~Kat{\=o}.
\newblock {\em {Perturbation theory for linear operators}}.
\newblock Grundlehren der mathematischen Wissenschaften. Springer, 1995.

\bibitem{Picard}
R.~Picard.
\newblock {A structural observation for linear material laws in classical
  mathematical physics.}
\newblock {\em Math. Methods Appl. Sci.}, 32(14):1768--1803, 2009.

\bibitem{Picard_McGhee}
R.~Picard and D.~McGhee.
\newblock {\em {Partial differential equations. A unified Hilbert space
  approach.}}
\newblock {de Gruyter Expositions in Mathematics 55. Berlin: de Gruyter.
  xviii}, 2011.

\bibitem{Picard2013_fractional}
R.~Picard, S.~Trostorff, and M.~Waurick.
\newblock {On Evolutionary Equations with Material Laws Containing Fractional
  Integrals.}
\newblock {\em Math. Methods Appl. Sci.}, 2014.
\newblock \href {http://dx.doi.org/10.1002/mma.3286}
  {\path{doi:10.1002/mma.3286}}.

\bibitem{Picard2014_survey}
R.~Picard, S.~Trostorff, and M.~Waurick.
\newblock {Well-posedness via Monotonicity. An Overview.}
\newblock In W.~Arendt, R.~Chill, and Y.~Tomilov, editors, {\em Operator
  Semigroups Meet Complex Analysis, Harmonic Analysis and Mathematical
  Physics}, volume 250 of {\em Operator Theory: Advances and Applications}.
  Birkh\"auser, 2015.
\newblock to appear.

\bibitem{picard1989hilbert}
R.~H. Picard.
\newblock {\em {Hilbert space approach to some classical transforms}}.
\newblock Pitman research notes in mathematics series. Longman Scientific \&
  Technical, 1989.

\bibitem{Pruss1984}
J.~{Pr{\"u}ss}.
\newblock {On the spectrum of $C\sb 0$-semigroups.}
\newblock {\em {Trans. Am. Math. Soc.}}, 284:847--857, 1984.

\bibitem{Pruss2009}
J.~Pr{\"u}ss.
\newblock {Decay properties for the solutions of a partial differential
  equation with memory.}
\newblock {\em Arch. Math.}, 92(2):158--173, 2009.

\bibitem{rudin1987real}
W.~Rudin.
\newblock {\em {Real and complex analysis}}.
\newblock Mathematics series. McGraw-Hill, 1987.

\bibitem{Trostorff2013_stability}
S.~Trostorff.
\newblock {Exponential Stability for Linear Evolutionary Equations.}
\newblock {\em Asymptotic Anal.}, 85:179--197, 2013.

\bibitem{Trostorff2014_PAMM}
S.~Trostorff.
\newblock {A Note on Exponential Stability for Evolutionary Equations.}
\newblock {\em Proc. Appl. Math. Mech.}, 14:983--984, 2014.

\bibitem{Trostorff2012_integro}
S.~Trostorff.
\newblock {On Integro-Differential Inclusions with Operator-valued Kernels.}
\newblock {\em Math. Methods Appl. Sci.}, 38(5):834--850, 2015.
\newblock \href {http://dx.doi.org/10.1002/mma.3111}
  {\path{doi:10.1002/mma.3111}}.

\bibitem{Trostorff2012_eeliptic}
S.~{Trostorff} and M.~{Waurick}.
\newblock {A note on elliptic type boundary value problems with maximal
  monotone relations.}
\newblock {\em {Math. Nachr.}}, 287(13):1545--1558, 2014.

\bibitem{Waurick2013_continuous_dep}
M.~Waurick.
\newblock Continuous dependence on the coefficients for a class of
  non-autonomous evolutionary equations.
\newblock Technical report, TU Dresden, 2013.
\newblock arXiv: 1308.5566.

\bibitem{Weiss1991}
G.~{Weiss}.
\newblock {Representation of shift-invariant operators on $L\sp 2$ by
  $H\sp{\infty}$ transfer functions: An elementary proof, a generalization to
  $L\sp p$, and a counterexample for $L\sp{\infty}$.}
\newblock {\em {Math. Control Signals Syst.}}, 4(2):193--203, 1991.

\end{thebibliography}
\end{document}